\documentclass[12pt,tbtags,leqno]{amsart}

\usepackage{cite}

\usepackage{graphicx}
\usepackage{yfonts}

\usepackage[T1]{fontenc}
\usepackage[latin1]{inputenc}
\usepackage[english]{babel}
\usepackage{amssymb}
\usepackage{amsthm}
\usepackage{amsmath}
\usepackage{verbatim}
\usepackage{color}
\usepackage{enumitem}

\numberwithin{equation}{section}

\newcommand{\HH}{\mathcal{H}}
\newcommand{\HK}{\mathcal{K}}
\newcommand{\HM}{\mathcal{M}}
\newcommand{\HE}{\mathcal{E}}

\newcommand{\HS}{\mathcal{S}}

\newcommand{\HD}{\mathcal{D}}
\newcommand{\HB}{\mathcal{B}}
\newcommand{\HR}{\mathcal{R}}
\newcommand{\HN}{\mathcal{N}}

\newcommand{\D}{\mathbb{D}}
\newcommand{\B}{\mathbb{B}}
\newcommand{\C}{\mathbb{C}}
\newcommand{\N}{\mathbb{N}}

\newcommand{\T}{\mathbb{T}}

\newcommand{\ran}{\mathrm{ran \ }}

\newcommand{\rank}{\mathrm{rank \ }}

\newcommand{\la}{\langle}
\newcommand{\ra}{\rangle}

\newcommand{\Hol}{\operatorname{Hol}}

\newcommand{\clos}{\operatorname{clos}}
\newcommand{\card}{\operatorname{card}}

\def\HE{{\mathcal E}}

\newcommand{\Ker}[1]{\mathsf{Ker}~}

\theoremstyle{plain}
\newtheorem{theorem}{Theorem}[section]
\newtheorem{lemma}[theorem]{Lemma}

\newtheorem{remark}[theorem]{Remark}
\newtheorem{corollary}[theorem]{Corollary}

\theoremstyle{definition}
\newtheorem{example}[theorem]{Example}

\author{Shuaibing Luo, Caixing Gu, and Stefan Richter}

\address{School of Mathematics, Hunan University, Changsha, Hunan, 410082, PR China}
\email{sluo@hnu.edu.cn}

\address{Department of Mathematics, California Polytechnic State University, San Luis Obispo, CA 93407, USA}
\email{cgu@calpoly.edu}

\address{Department of Mathematics, University of Tennessee, Knoxville, TN, 37996, USA}
\email{srichter@utk.edu}
\thanks{S. Luo was supported by NNSFC (\# 11701167).}
\date{\today}
\begin{document}

\begin{abstract} We investigate expansive Hilbert space operators $T$ that are finite rank perturbations of isometric operators. If the spectrum of $T$ is contained in the closed unit disc $\overline{\mathbb{D}}$, then  such operators are of the form $T= U\oplus R$, where $U$ is isometric and $R$ is unitarily equivalent to the operator of multiplication by the variable $z$ on a de Branges-Rovnyak space $\mathcal{H}(B)$. In fact, the space $\mathcal{H}(B)$  is defined in terms of a rational operator-valued Schur function $B$. In the case when $\dim \ker T^*=1$, then $\mathcal{H}(B)$ can be taken to be a space of scalar-valued analytic functions in $\mathbb{D}$, and the function $B$ has a mate $a$ defined by $|B|^2+|a|^2=1$ a.e. on $\partial \mathbb{D}$. We show the mate $a$ of a rational $B$ is of the form $a(z)=a(0)\frac{p(z)}{q(z)}$, where $p$ and $q$ are appropriately derived from the characteristic polynomials of two associated operators. If $T$ is a $2m$-isometric expansive operator, then all zeros of $p$ lie in the unit circle, and we completely describe the spaces $\mathcal{H}(B)$ by use of what we call the local Dirichlet integral of order $m$ at the point $w\in \partial \mathbb{D}$.

\end{abstract}

\date{\today}
\title[Higher order local Dirichlet integrals]{Higher order local Dirichlet integrals and de Branges-Rovnyak spaces}

\maketitle

\tableofcontents

\section{Introduction}
Let $\mathbb{D}$ be the open unit disc in the complex plane $\mathbb{C}$, and $\mathbb{T}=\partial \D$ be the unit circle. If $b:\D\to \D$ is analytic, then the de Branges-Rovnyak space $\HH(b)$ is the unique Hilbert space of holomorphic functions on $\D$ with reproducing kernel $$K^b_w(z)=\frac{1-b(z)\overline{b(w)}}{1-z\overline{w}},$$ i.e. $K^b$ satisfies $f(w)=\la f, K^b_w\ra$ for all $f\in \HH(b)$.  De Branges-Rovnyak spaces possess a rich structure, and they play an important role in many aspects of complex analysis and operator theory. We refer the reader to the books of de Branges and Rovnyak \cite{dBR66}, Sarason \cite{Sa94},  Fricain and Mashreghi \cite{FM16}, and for some recent developments to  \cite{AlemanMalman, BBB15, BFM10, BallBoloBasics, BFGHR15,CGR10,CR13,EFKMR16,FHR16,FKM19,KZ15,LN19,MT14} and the references therein.

It is well-known that the backward shift $Lf(z)=\frac{f(z)-f(0)}{z}$ acts contractively on every de Branges-Rovnyak space, but the forward shift $(M_z,\HH(b))$ defined by $(M_zf)(z)=zf(z)$ is bounded  only if $b$ is not an extreme point of the unit ball of $H^\infty$, see \cite{Sa94}. This condition is known to be equivalent to the existence of a {\it mate} for $b$, i.e. an outer function $a$ such that  $|a|^2+|b|^2=1$ a.e. on the unit circle $\T$. The mate is unique, if we also assume that $a(0)>0$. In this paper we will make this assumption, and then refer to the unique mate as {\it the} mate of $b$.

Thus, if such a forward shift is bounded, then it expands the norm. In this paper we investigate which expansive operators $T$ are unitarily equivalent to $(M_z,\HH(b))$ for rational functions $b$, and we obtain further results that link properties of $T$ and of $b$. It will follow from Lemma \ref{unitaryBalpha} that, if for $\alpha \in \D$ $b_\alpha(z)=\frac{\alpha-b(z)}{1-\overline{\alpha}b(z)}$, then $(M_z,\HH(b))$ is unitarily equivalent to $(M_z,\HH(b_\alpha))$. Hence there will be no loss of generality in assuming that $b(0)=0$.

If $b$ is a rational function, then the degree of $b$ is defined to be the larger of the degrees of the polynomials $p$ and $q$ provided that $b=\frac{p}{q}$ is in reduced form. In \cite{Sa97} Sarason observed that if $b$ is a certain rational function of degree 1, then $\HH(b)$ equals a so-called local Dirichlet space. Local Dirichlet spaces are important for the study of two-isometric operators, i.e. bounded linear operators $T\in \HB(\HH) $ that satisfy ${T^*}^2T^2-2T^*T+I=0$, see e.g. \cite{Ri88, Ri91, RS91}. Sarason's result has been refined and extended in \cite{CGR10, CR13,KZ15,EFKKMR16}, and one of their   results can be paraphrased as follows: If $b(0)=0$, then $(M_z,\HH(b))$ is a two-isometry, if and only if \begin{align}\label{deg 1 b}b(z)= e^{it} \frac{(1-r)\overline{w}z}{1-r\overline{w}z}\end{align} for some $t\in [0,2\pi)$, $0<r\le 1$ and $|w|=1$. In this case $\|f\|^2_{\HH(b)}=\|f\|^2_{H^2}+ \frac{(1-r)^2}{r} D_w(f)$, where $D_w(f)=\int_{|z|=1}\left|\frac{f(z)-f(w)}{z-w}\right|^2\frac{|dz|}{2\pi}$ is the local Dirichlet integral of $f$, see Theorem 3.1 of \cite{CGR10}.

Let $m\in \N$. An operator $T\in \HB(\HH)$ is called an $m$-isometry, if
$$\beta_m(T)=\sum\limits_{k=0}^{m}(-1)^{m-k}\binom{m}{k}T^{\ast k}T^{k}=0.$$
The study of $m$-isometries originated in the work of Agler \cite{Ag90}, Condition (2.7), and it was at least partially motivated by an analogy with Helton's study of $m$-symmetric operators, \cite{He1972}. The first in-depth study of $m$-isometries was carried out in \cite{AS956}, and by now there is an extensive list of references for these operators.

Since $\beta_{m+1}(T)=T^*\beta_m(T)T-\beta_m(T)$ it follows that every $m$-isometry is a $k$-isometry for every $k\ge m$, and we say that $T$ is a strict $m$-isometry, if $\beta_{m-1}(T)\ne 0$. If $m\in \N$, $w\in \T$, and if $f\in H^2$ extends to be analytic in a neighborhood of $w$, then we define the local Dirichlet integral of $f$ of order $m$ at $w$ by
$$D^m_w(f)= \int_{|z|=1} \left|\frac{f(z)-T_{m-1}(f,w)(z)}{(z-w)^m}\right|^2 \frac{|dz|}{2\pi},$$ where $T_{m-1}(f,w)$ is the $(m-1)$-th order Taylor polynomial of $f$ at $w$. This definition can be extended to apply to more general functions, see Section \ref{SectionLocalDiri}. In particular, we note that  $D^1_w(f)=D_w(f)$ for all $f$. The space $\HD_w^m$ is defined to consist of all $f\in H^2$ such that $D_w^m(f)<\infty$. We will show

\begin{theorem} \label{2n-iso} Let $b$ be a non-extreme point of the unit ball of $H^\infty$ with $b(0)=0$, and let $m\in \N$. Then
 $(M_z,\HH(b))$ is not a strict $(2m+1)$-isometry, and the following are equivalent:
 \begin{enumerate}
\item $(M_z,\HH(b))$ is a strict $2m$-isometry,
\item  $b$ is a rational function of degree $m$ such that the mate has a single zero of multiplicity $m$ at a point $w \in \T$,
\item there is a $w\in \T$ and a polynomial $\tilde{p}$ of degree $< m$ with $\tilde{p}(w)\ne0$  such that $\|f\|^2_{\HH(b)}=\|f\|^2_{H^2}+ D_w^m(\tilde{p}f)$.
 \end{enumerate}

 If the three conditions hold, then there are polynomials $p$ and $q$ of degree $\le m$ such that $b=\frac{p}{q}$, $a=\frac{(z-w)^m}{q}$, $\tilde{p}(z)=z^m\overline{p(1/\overline{z})}$ for $z\in \D$, and $|q(z)|^2= |p(z)|^2+|z-w|^{2m}$ for all $z\in \T$. Furthermore, $\HH(b)=\HD_w^m$ with equivalence of norms.
\end{theorem}

Note that if $m=1$ in this theorem, then $|b(w)|=1$ for some $w\in \T$, and then the conditions that $b(0)=0$ and that $b$ has degree 1 imply that it has the form as in (\ref{deg 1 b} ) for $r<1$. And in fact, one checks that for two-isometries this theorem is equivalent to the earlier one.

Theorem \ref{2n-iso} raises two questions:

1. If $b$ is a more general rational function and non-extreme, then what operator properties of the forward and backward shifts  on $\HH(b)$ can be easily seen by looking at  the  mate $a$ of $b$?

2. Among all $m$-isometries, how general a class are the ones that are unitarily equivalent to $(M_z,\HH(b))$ for some non-extreme $b$?

We give answers to both questions. In fact, it will be instructive to consider the vector-valued de Branges-Rovnyak spaces $\HH(B)$, where  $B\in \HS(\HD,\HE)$ is an operator-valued Schur class function, i.e. $B:\D\to \HB(\HD,\HE)$ is a contractive analytic function. Our results are most complete in the interesting case where $H(B)$ is a space of scalar-valued analytic functions, i.e. when $\HE=\C$.

An operator $T$ is called analytic, if $\bigcap_{n\ge 0} \ran T^n=(0)$. It turns out that by combining a construction of Shimorin's (\cite{Sh01}) and theorems of Aleman and Malman \cite{AlemanMalman} one sees rather easily that a Hilbert space operator $T$  is norm expansive and  analytic if and only if it is unitarily equivalent to $(M_z,\HH(B))$ for some operator-valued Schur class function $B$ with $B(0)=0$. We will need details from this construction, thus we have included the complete details in Section \ref{Operator deBranges}.

For $T\in \HB(\HH)$ write $\Delta=T^*T-I$. In particular, we will see that an operator $T$ is unitarily equivalent to $(M_z,\HH(b))$ for some non-extremal $b$ in the unit ball of $H^\infty$, if and only if $T$ is analytic, $\dim\ker T^*=1$, and $\Delta$ is positive and has rank 1. In the more general case where $T$ is analytic, $\dim \ker T^*=1$ and $\Delta\ge 0$ one can take $B=(b_1, b_2, \dots)$ for scalar-valued analytic functions $b_j$ and one obtains a space $\HH(B)$ of scalar-valued analytic functions with reproducing kernel $K^B_w(z)= \frac{1-\sum_{i=1}^\infty b_i(z)\overline{b_i(w)}}{1-z\overline{w}}$. The minimum number of functions $b_j$ that are not identically equal to 0 equals the rank of  $ \Delta$, see the beginning of Section 7. As before, the condition $B(0)=0$ can be seen as a normalization that assures that $\HH(B)$ contains the constant functions. This set-up applies for example to $T= \sqrt{2}(M_z, L^2_a)$, where $(M_z,L^2_a)$ is the forward shift on the Bergman space of the unit disc. Of course, in this case $\HH(B)$ will be a space of analytic functions on a  disc of radius $\sqrt{2}$, and one may not expect to easily obtain deep information by looking at the corresponding Schur function $B$. However, for many interesting expansive operators $T$ the defect operator $\Delta = T^*T-I$ will be compact. That is true for example for the Dirichlet shift, and more generally, if $T=M_z$ on a superharmonically weighted Dirichlet space (see \cite{LR15}, Theorem 5.1). If this is the case, then $\Delta=\sum_{n\ge 1} t_n f_n\otimes f_n$ for some $0<t_n\to 0$ and an orthonormal basis $\{f_n\}$ of $\overline{\ran \Delta}$, and one can show that for the $B$ one can take $b_n/z=\sqrt{\frac{t_n}{1+t_n}}\ f_n$. This follows from the proofs of Lemmas \ref{rangeofDelta and D*} and \ref{rangeofDelta}.

In this paper we will be interested in the situation, where $\Delta$ is a finite rank operator. For that case we can take $B=(b_1, \dots, b_k)$ and Aleman and Malman proved that $M_z$ acts boundedly on $\HH(B)$, if and only if $1-\sum_{i=1}^k|b_i|^2$ is log-integrable on $\T$, \cite{AlemanMalman}. Thus, as in the classical rank 1 case such $B$ will have a mate $a$, the unique outer function with $a(0)>0$ and $|a|^2 + \sum_{i=1}^k|b_i|^2=1$ a.e. on $\T$.

Crucial to all our results will be the space $\HN=[\ran \Delta]_{T^*}$, the smallest $T^*$-invariant subspace that contains $\ran \Delta$. It is not difficult to show that $\HM=\HN^\perp$ is the largest $T-$invariant subspace such that $T|\HM$ is isometric. We will show that $\HN=[\ran \Delta]_L$, the smallest $L$-invariant subspace that contains $\ran \Delta$, see Lemma \ref{T*Linvariant}. Here $Lf(z)=\frac{f(z)-f(0)}{z}$ is the backward shift. Of course, this space may be all of $\HH(B)$. Our  theorem describes a case when this does not happen. $B$ is called rational, if each $b_i$ is a rational function. By taking common denominators it follows that rational $B$ are of the form $B=(p_1/q,\dots, p_k/q)$, where $q, p_1,\dots p_k$ are polynomials. The degree of $B$ is defined to be the smallest $n$ such that $B$ has such a representation where the degrees of all polynomials are less than or equal to $n$. So for example, $B(z)=\left(\begin{matrix} \frac{1}{5+z} \ \ \frac{1}{6+z}\end{matrix}\right)$ has degree 2. Recall that the characteristic polynomial $p$ of an $n\times n$ matrix $A$ is defined by $p(z)=\det(zI-A)$.

\begin{theorem}\label{characterizationmate}
Let $B=(b_1,  \dots, b_k)$ be such that $B(0)=0$ and $1-\sum_{i=1}^k|b_i|^2$ is log-integrable on $\T$. Let $T = (M_z, \HH(B))$ and $\Delta = T^*T - I$.

Then $B$ is a rational function of degree $n$, if and only if $$\dim [\ran \Delta]_{T^*}=n<\infty.$$

Furthermore, if  $ \HN=[\ran \Delta]_{T^*}$ has dimension $n <\infty$, and if $p(z)=\prod_{i=1}^n(z-w_i)$ is the characteristic polynomial of $T^*|\HN$ and $q(z)=\prod_{i=1}^n(z-\alpha_i)$ is the characteristic polynomial of $L|\HN$, then  $$a(z)=a(0) \frac{\prod_{i=1}^n(1-w_iz)}{\prod_{i=1}^n(1-\alpha_i z)} .$$
\end{theorem}

We approach the second question by first stating a theorem for more general expansive operators with finite rank defect $\Delta$.
\begin{theorem}\label{representationVR}
Let $T\in \HB(\HH)$ with $\Delta=T^*T-I\ge 0$, and let $\HN=[\ran \Delta]_{T^*}$.

Then the following are equivalent:

(a) $\HN$ is finite dimensional and $\sigma(T)\subseteq \overline{\D}$,

(b) $T=V\oplus R$, where $V$ is isometric and $R$ is unitarily equivalent to $(M_z,\HH(B))$ for some rational $B\in \HS(\C^k,\C^m)$.

If (a) and (b) are satisfied, then one can choose $k=\dim \ran \Delta$, $m= \dim \ran (I-P_{\HN})TP_{\HN}=\dim \ker R^*$, and one has
$$ \text{degree }B \le \dim \HN \le m\text{ degree }B.$$
\end{theorem}
The degree of a rational matrix-valued function $B(z) = (b_{ij}(z))_{ij}$ is defined analogously to the way it was defined for row operator-valued functions.

Two-isometric operators are automatically norm expansive. That is no longer true for $m$-isometries, if $m\ge 3$. For example, one easily checks that $T=\left[\begin{matrix}1&\alpha\\0&1\end{matrix}\right], \alpha \neq 0$ defines a 3-isometric operator on $\C^2$ that is not norm expansive. Still, for each $m\in \N$ there are examples of norm expansive strict $m$-isometries, see e.g. \cite{GuLuo18}. For norm expansive $m$ isometries with finite rank defect operator $\Delta$ there are some restrictions.
\begin{theorem}\label{twonisometry} Let $T\in \HB(\HH)$ be such that $\Delta = T^*T-I$ is a positive operator of finite rank, and let $m\in \N$.

(a) If $T$ is a $2m+1$-isometry , then  $T$ is a $2m$-isometry.

(b) $T$ is a $2m$-isometry, if and only if there are $w_1, \dots, w_k\in \T$ and positive operators $\Delta_1,\dots, \Delta_k$ such that $\Delta=\sum_{j=1}^k \Delta_j$ and $(T^*-\overline{w}_j)^m \Delta_j=0$ for each $j$.

Furthermore, if the above is satisfied with $\Delta_j \ne 0$ for all $j$ and if $\HN=[\ran \Delta]_{T^*},$ then $\HN$ is finite dimensional and  $\sigma(T^*|\HN)=\{\overline{w}_1, \dots, \overline{w}_k\}$.
\end{theorem}

Theorem \ref{representationVR} implies that if one wants to classify the expanding $2m$-isometries with finite rank $\Delta$, then one needs to understand the $2m$-isometries acting on de Branges-Rovnyak spaces $\HH(B)$. It turns out that if $B\in \HS(\C^n,\C)$, then $T=(M_z,\HH(B))$ is a $2m$-isometry, if and only if $\HH(B)$ can be represented as an intersection of finitely many rank 1 spaces $H(b_i)$ as considered in Theorem \ref{2n-iso}.

\begin{theorem}\label{MainTheorem} Let $B\in \HS(\C^N,\C)$ be a contractive analytic function with $B(0)=0$ and such that $T=(M_z,\HH(B))$ is bounded and satisfies $\rank \Delta <\infty$. Let $\HN=[\ran \Delta]_{T^*}$.

Then $T$ is a $2m$-isometry, if and only if there are $w_1, \dots, w_k \in \T$, pairs of integers $(m_1,n_1), \dots (m_k, n_k)$ such that $1 \le n_j\le m_j \le m$ for all $j$,   and there are polynomials $\{p_{ij}\}_{1\le j\le k, 1\le i\le n_j}$ with degree $p_{ij} \le m_j-1$ for $1\le j \le k$, $1\le i\le n_j$  such that

\begin{align}\label{2mIsoNorm}\|f\|^2_{\HH(B)}=\|f\|^2_{H^2} + \sum_{j=1}^k\sum_{i=1}^{n_j} D^{m_j}_{w_j}(p_{ij}f).\end{align}

There is a choice of parameters so that for each $j$ there is an $i$ with $p_{ij}(w_j)\ne 0$ and such that $\sum_{j=1}^kn_j=\rank \Delta$. If all that is the case, then $\HH(B)=\bigcap_{j=1}^k \HD_{w_j}^{m_j}$ with equivalence of norms, the characteristic polynomial of $A=P_\HN T|\HN$ is $p_A(z)=\prod_{j=1}^k(z-w_j)^{m_j}$,
and the mate $a$ of $B$ is of the form $a(z)= \frac{p_A(z)}{q(z)}$, where $q$ is the unique polynomial of degree $\le \sum_{j=1}^k m_j$, which has no zeros in $\overline{\D}$, and satisfies $\frac{p_A(0)}{q(0)}>0$ and $$|q(z)|^{2}= |p_A(z)|^2+ \sum_{j=1}^k \left|\frac{p_{A}(z)}{(z-w_j)^{m_j}}\right|^2\sum_{i=1}^{n_j}|p_{ij}(z)|^2  \ \ \text{ for all }|z|=1.$$
\end{theorem}
In the scalar and rank 1 case  a formula for the norm of a function $f\in \HH(b)$ is given by $\|f\|^2_{\HH(b)}=\|f\|^2_{H^2}+\|f^+\|^2_{H^2}$, where $f^+$ is a function that is appropriately associated with $f$, see \cite{Sa94}, section IV-1. Formula (\ref{2mIsoNorm}) is the analogue of this for the case considered in the theorem. An important ingredient to derive (\ref{2mIsoNorm}) is a formula of Aleman-Malman, see Lemma \ref{AlemanMalmanFormula}.

In the above we mentioned the following useful boundedness criterion of Aleman and Malman.
\begin{theorem} \label{AleMalTheorem}(\cite{AlemanMalman}, Theorem 5.2) If $B=(b_1, \dots, b_n)\in \HS(\C^n,\C)$, then $(M_z,\HH(B))$ is bounded, if and only if $$\int_{|z|=1} \log (1-\sum_{i=1}^n |b_i(z)|^2 )\frac{|dz|}{2\pi}>-\infty.$$
\end{theorem}
Simple examples show that this condition does not capture the complete characterization for boundedness in the general case, see Example \ref{DiriExample}. It is thus worthwhile to point out that routine methods establish the following necessary and sufficient condition, see Lemma \ref{condfsfwsb}.

\begin{theorem}\label{condfsfwsb}
Let $B\in \HS(\HD,\HE)$ with $B(0)=0$,  then  the following are equivalent:
\begin{enumerate}
\item $T: f\to zf$ defines a bounded operator on $\HH(B)$,
\item for each $x\in \HD$ the function $g_x(z)=B(z)x\in \HH(B)$.
\end{enumerate}
\end{theorem}

The paper is organized as follows. In Section 3 we consider general Hilbert space operators $T\in \HB(\HH)$ such that $\Delta=T^*T-I\ge 0$ and that have the additional property that $\HN=[\ran \Delta]_{T^*}$ is finite dimensional. In Theorem \ref{Wold-type} we will establish the implication (a) $\Rightarrow $ (b) of Theorem \ref{representationVR}. The reverse implication and the rest of Theorem \ref{representationVR} will follow from Theorem \ref{ranDelta and rational}. In Section 4 we present our approach to the fact that all expansive analytic operators can be represented as $M_z$ on a de Branges-Rovnyak space $\HH(B)$ for some operator-valued Schur function $B$, see Theorem \ref{expandingOp}. The theorem says that one may always assume that $B(0)=0$. If $B$ and $C$ are two scalar-valued Schur functions with $B(0)=C(0)=0$, then it turns out that the operators $(M_z,\HH(B))$ and $(M_z,\HH(C))$ are unitarily equivalent, if and only if $\HH(B)=\HH(C)$ with equality of norms, see Lemma \ref{unitaryEquiv}. Section 5 contains some background facts about general Schur functions with $B(0)=0$. We already mentioned that in Theorem \ref{ranDelta and rational} we will establish the remaining parts of Theorem \ref{representationVR}. Additionally, we show that any rational Schur function $B$ is of the form $B(z)=\frac{1}{\tilde{q}(z)}P(z)$ for some operator polynomial $P$ and the scalar function $\tilde{q}(z)= z^nq(1/z)$, where $q$ is the characteristic polynomial of $L|\HN$ and $n=\dim \HN$. In Section 7 we restrict attention to scalar-valued Schur functions, and  Theorem \ref{characterizationmate} will follow from Theorems \ref{ranDelta and rational} and \ref{rationalmate}. In Sections 8-10 we prove our theorems about  expansive $n$-isometric operators. The proof of Theorem \ref{DeltajImplies2miso} contains an elementary argument that will establish  one of the implications in part (b) of Theorem \ref{twonisometry}. The main ingredient for the remainder of Theorem \ref{twonisometry} will be a construction in which we  put a new norm on $\HN$ in such a way that it allows us to use results of Agler-Helton-Stankus (\cite{AHS98}) about $n$-isometries on finite dimensional spaces. In Section 9 we give the rigorous definition of the higher order local Dirichlet integral and we prove some general facts about the spaces $\HD_w^m$. We also establish Theorem \ref{Thm1.1again}, which is a version of Theorem \ref{2n-iso}. Corollary \ref{equivalentNorms} says that for a fixed $w\in \T$ all spaces $\HH(b)$ as defined by (iii) of Theorem \ref{2n-iso} equal $\HD_w^m$ with equivalence of norms. In Section 10 we present the proof of Theorem \ref{MainTheorem}, see Theorems \ref{IntersectionOfBasic} and \ref{MainThmConverse}. Finally, in Section 11 we outline method of how to calculate the Schur function $B$ in Theorem \ref{MainTheorem}, if one knows the polynomials $p_{ij}$ and the rank of $\Delta$.

\section{Preliminaries} We start with a lemma that is probably well-known, we include the proof for completeness. Recall from the Introduction that $T\in \HB(\HH)$ is called analytic, if $\bigcap_{n\ge 1} \ran T^n =(0)$. This is equivalent to $\bigvee_{n\ge 1} \ker {T^*}^n=\HH$.
\begin{lemma} \label{kernels and analyticity} Let $0<\varepsilon<1$ and let $T\in \HB(\HH)$ with $\|Tx\|\ge \|x\|$ for all $x\in \HH$. Then $$\bigvee_{n\ge 1} \ker {T^*}^n = \bigvee_{|\lambda|<\varepsilon} \ker (T^*-\overline{\lambda}).$$
\end{lemma}
\begin{proof} Since $T$ is bounded below, the operator $T^*T$ is invertible. We set $L=(T^*T)^{-1}T^*$, the left inverse of $T$ with $\ker T^*=\ker L$. Note that $TL$ is a projection, thus $\|Lx\|\le \|TLx\|\le \|x\|$. Hence $(1-\overline{\lambda} L^*)^{-1}$ exists for all $|\lambda|<1$.

Fix $|\lambda|<\varepsilon$ and let $x \in \ker(T^*-\overline{\lambda})$. Set $y=(1-\overline{\lambda} L^*)x$, and observe that $y\in \ker T^*$ since $T^*L^*=I$. For $N\in \N$ let $x_N=\sum_{k=0}^N \overline{\lambda}^k{L^*}^ky$, then ${T^*}^{N+1}x_N=\sum_{k=0}^N \overline{\lambda}^k{T^*}^{N+1-k}y=0$. Thus $x_N \in \bigvee_{n\ge 1} \ker {T^*}^n $ and $x_N\to (1-\overline{\lambda} L^*)^{-1}y=x$. This implies that $$\bigvee_{|\lambda|<\varepsilon} \ker (T^*-\overline{\lambda})\subseteq \bigvee_{n\ge 1} \ker {T^*}^n.$$
In order to show the reverse inclusion, we set $\HM =\bigvee_{|\lambda|<\varepsilon} \ker (T^*-\overline{\lambda})$, and we will start by showing that $\HM$ is invariant for $L^*$. First we let $0<|\lambda|<\varepsilon$ and  $x\in \ker(T^*-\overline{\lambda})$. Then $$L^*x=\frac{1}{\overline{\lambda}}x-\frac{1}{\overline{\lambda}}(1-\overline{\lambda}L^*)x$$ is a difference of an element in $\ker(T^*-\overline{\lambda})$ and an element in $\ker T^*$. Thus, $L^*x\in \HM$ whenever $x\in \ker(T^*-\overline{\lambda})$ for $\lambda\ne 0$. If $x\in \ker T^*$, then for small $|\lambda|\ne 0$ we have
$$y_\lambda=(1-\overline{\lambda}L^*)^{-1}x \in \ker(T^*-\overline{\lambda}).$$
Then $L^*y_\lambda\in \HM$, and $L^*x\in \HM$ follows from the fact that $L^*y_\lambda\to L^*x$ as $|\lambda|\to 0$. Thus, $\HM$ is $L^*$-invariant.

We will now use induction on $n$ to show that $\ker {T^*}^n\subseteq \HM$ for each $n\ge 1$. The conclusion is obvious, if $n=1$. Suppose that $n\ge 1$ and that $\ker {T^*}^n\subseteq \HM$, and let $x\in  \ker {T^*}^{n+1}$. Then $T^*x\in \ker {T^*}^n\subseteq \HM$, and hence by the $L^*$-invariance $L^*T^*x\in \HM$. Then $x=(I-L^*T^*)x+L^*T^*x$ is a sum of elements from $\ker T^*\subseteq \HM$ and $\HM$, hence $x\in \HM$.
\end{proof}
\begin{corollary}\label{analytic Corollary} Let $T\in \HB(\HH)$ with $\|Tx\|\ge \|x\|$ for all $x\in \HH$, and let $S\subseteq \D$ be a set that has an accumulation point in $\D$. Then $$\bigvee_{n\ge 1} \ker {T^*}^n = \bigvee_{\lambda \in S} \ker (T^*-\overline{\lambda}).$$
\end{corollary}
\begin{proof} By the previous lemma (and by taking orthocomplements) it suffices to show that
$$\bigcap_{\lambda\in S}\ran(T-\lambda)\subseteq \bigcap_{\lambda\in \D}\ran(T-\lambda).$$

Let $x\in \bigcap_{\lambda\in S}\ran(T-\lambda)$, and let $L$ be the left inverse of $T$ as in the previous proof. Then $L$ is a contraction and we note that $(1-\lambda L)^{-1}x\in \ran T$ for every $\lambda \in S$. Let $P$ be the projection onto $\ran T^\perp$. Then the analytic function $F(z)=P(1-zL)^{-1}x$ is zero on $S$, hence it is identically equal to zero in $\D$. Thus, $(1-zL)^{-1}x\in \ran T$ for all $z\in \D$, and that implies $x\in \ran(T-z)$ for every $z\in \D$.
\end{proof}

\section{Finite rank extensions}\label{Finite rank extensions}
\begin{lemma}\label{unitary summand} Let $R\in \HB(\HH)$ with $R^*R\ge I$. If $\HM\subseteq \HH$ is an invariant subspace of $R$ such that $U=R|\HM$ is unitary, then $\HM$ is a reducing subspace for $R$.
\end{lemma}
\begin{proof} We have $$R=\left[\begin{matrix} U &C\\0&A\end{matrix}\right]$$ and we calculate that $$R^*R-I=\left[\begin{matrix} 0 &U^*C\\C^*U&C^*C+A^*A-I\end{matrix}\right]\ge 0.$$ The positivity implies that $U^*C=0$, hence $C=0$.
\end{proof}

\begin{lemma} \label{rank1extension} Let $\HH$ be a Hilbert space, $T\in \HB(\HH)$, $\lambda\in \overline{\D}$,  $c_0\in \HH$, and $$R=\left[\begin{matrix} T &C\\0&\lambda\end{matrix}\right]  \text{ on }\HH\oplus\C,$$ where $C:\C\to \HH$ satisfies $C1=c_0$.

If $T$ is analytic and if $R^*R\ge I$, then either $R$ is analytic or $|\lambda|=1$ and there is an analytic operator $R'\in \HB(\HH')$ such that $R$ is unitarily equivalent to $R' \oplus \lambda$ acting on $\HH'\oplus \C$.
\end{lemma}
\begin{proof} Suppose $R$ is not analytic and define $S=\D\setminus\{\lambda\}$ and  $\HH'=\bigvee_{z\in S} \ker (R^*-\overline{z})$. Then by Corollary \ref{analytic Corollary} we have $\HH'^\perp = \bigcap_{z\in S} \ran(R-z) \ne 0$. If $x\in \HH$ such that $x\oplus 0\in \HH'^\perp$, then the hypothesis that $T$ is analytic implies that $x=0$. Thus, $\HH'^\perp$ is 1-dimensional and it contains an element of the type $x\oplus 1$, where $x\in \HH$. We claim that $c_0=(\lambda-T)x$.

Let $z\in S$. Then there are $u\in \HH$ and $\alpha \in \C$ such that $x=(T-z)u+\alpha c_0$ and $1=(\lambda-z)\alpha$. Hence with some algebra we obtain $$c_0-(\lambda-T)x=(T-z)((z-\lambda)u+x)\in \ran(T-z)$$ for all $z\in S$. The analyticity of $T$ implies that $c_0-(\lambda-T)x=0$.

Thus, $R(x\oplus 1)= (Tx+c_0)\oplus \lambda= \lambda(x\oplus 1)$. Thus, $\lambda$ is an eigenvalue of $R$. Since $R^*R\ge I$ we must have $|\lambda|=1$. Then by Lemma \ref{unitary summand} the eigenspace $\C (x\oplus 1)$ is reducing for $R$.

Set $R'=R|\HH'$. It is clear that  $ \ker(R'^*-\overline{z})=\ker(R^*-\overline{z} )$ for all $z\in S$. Hence by Corollary \ref{analytic Corollary} and the definition of $\HH'$ it follows that $R'$ is analytic.
\end{proof}

In the following we will use the convention that any operator on a $0$-dimensional space is both unitary and analytic. Thus, when we consider operators of the type $R=U\oplus T$, where $U$ is unitary and $T$ is analytic, then this includes the cases where either summand may be absent.
\begin{theorem}\label{finite extension} Let $R\in \HB(\HH)$ such that $R^*R\ge I$ and $\sigma(R) \subseteq \overline{\D}$.

If $\HM \subseteq \HH$ is a $R$-invariant subspace of finite codimension and such that $T=R|\HM$ is analytic, then  $R=U\oplus T$, where $U$ is unitary and $T$ is analytic.
\end{theorem}
\begin{proof}  $R$ has a representation as a $2\times 2$ operator matrix of the following form
 $$R=\left[\begin{matrix}T&C\\0&A\end{matrix}\right] \ \text{ with respect to }\HH=\HM\oplus \HM^\perp.$$
Set $\HM_0=\HM$ and write $R_0=T$. Since $A$ acts on a finite dimensional space there is an orthonormal set $\HB=\{e_1, \dots, e_n\}$ such that the matrix for $A$ with respect to $\HB$ is in upper triangular form with $\lambda_1, \dots, \lambda_n$ on the diagonal. The complex conjugates of these numbers are eigenvalues for $A^*$ and hence they are eigenvalues for $R^*$, so we must have $\{\lambda_1, \dots, \lambda_n\}\subseteq \sigma(R)\subseteq \overline{\D}$.

For $1\le j\le n$ set $\HM_j=\HM_0 \oplus \text{ span}\{e_1, \dots ,e_j\}$ and $R_j=R|\HM_j$. We will show inductively that for each $j$ we have $R_j=U_j\oplus T_j$, where $U_j$ is unitary and $T_j$ is analytic. Of course, $R_j^*R_j\ge I$ for each $0\le j\le n$.

Since $R_0=T$ we know that $R_0$ is analytic, hence  $R_1=U_1\oplus T_1$ is of the required form by Lemma \ref{rank1extension}.   Suppose that $2\le j\le n$ and that either $R_{j-1}$ is analytic or $R_{j-1}=U_{j-1}\oplus T_{j-1}$, where $U_{j-1}$ is unitary and $T_{j-1}$  is analytic. If $R_{j-1}$ is analytic, then the conclusion follows from Lemma \ref{rank1extension}. Suppose now that $R_{j-1}=U_{j-1}\oplus T_{j-1}$. Then the space that $U_{j-1}$ acts on is invariant for $R_j$, hence by Lemma \ref{unitary summand} it is reducing for $R_j$. Hence $R_j$ is of the form
$$R_j=\left[\begin{matrix}U_{j-1}&0&0\\0&T_{j-1}&C_j\\0&0&\lambda_j\end{matrix}\right].$$ Thus, we may apply Lemma \ref{rank1extension} to $\left[\begin{matrix}T_{j-1}&C_j\\0&\lambda_j\end{matrix}\right]$ and conclude that it is either analytic or a direct sum of an analytic operator and a 1-dimensional unitary. In either case, it follows that $R_j$ is a direct sum of a unitary and an analytic operator. This concludes the induction and it proves the theorem since $R=R_n$.
\end{proof}

If $T^*T\ge I$, then we can form the smallest $T^*$-invariant subspace that contains the range of $\Delta =T^*T-I$. We write $\HN_T=[\ran \Delta]_{T^*}$ to denote this space. Note that the orthocomplement of this space is the largest $T$-invariant subspace $\HM$ such that $T|\HM$ is isometric.
\begin{theorem}\label{Wold-type} Let $T\in \HB(\HH)$ such that $T^*T\ge I$ and $\sigma(T) \subseteq \overline{\D}$.
If $\HN_T=[\ran \Delta]_{T^*}$ is finite dimensional, then $T=U\oplus T_1$, where $U$ is unitary and $T_1$ is analytic.

In fact, if $P$ denotes the projection onto $\HN_T$, then  $$T=V\oplus R,$$ where $V$ is isometric and $R$ is analytic with $\dim \ker R^*=\rank (I-P)TP<\infty$.
\end{theorem}
Let $V$ be isometric. Then by the classical Wold decomposition theorem $V=U\oplus S$, where $U$ is unitary, and $S$ is a unilateral shift (which is analytic). Then for any analytic operator $R$  the operator   $T_1=S\oplus R$ is analytic as well. Hence the first sentence of the theorem follows from the last sentence of the theorem.
\begin{proof} Let $\HM = \HN_T^\perp$. Then as remarked above $W=T|\HM$ is isometric, and $T$ has a $2 \times 2$ operator matrix representation of the type $$T=\left[\begin{matrix}W&C\\0&A\end{matrix}\right].$$ As in  the proof of Lemma \ref{unitary summand} we compute $$T^*T-I= \left[\begin{matrix}0&W^*C\\C^*W&C^*C+A^*A-I\end{matrix}\right]\ge 0,$$ and the positivity implies that $W^*C=0$. In this case that implies $\ran C \subseteq \ker W^*$. By the Wold decomposition theorem $W$ is a direct sum of a unitary operator $U$ and a unilateral shift. The unilateral shift part can be viewed as a direct sum of two unilateral shifts $S$ and $S'$ where $\ker S^*=\ran C$ (which is finite dimensional). Thus, $W=V\oplus S$ where $V=U\oplus S'$ is isometric. Hence $T$ has the following $3 \times 3$ operator matrix representation
$$T=\left[\begin{matrix}V&0&0\\0&S&C\\0&0&A\end{matrix}\right].$$
It now follows from Theorem \ref{finite extension} that $R=\left[\begin{matrix}S&C\\0&A\end{matrix}\right]= U_0\oplus R_0$ for a unitary $U_0$ and an analytic operator $R_0$. But by our construction $R$ does not have a nontrivial unitary direct summand. Hence it must be analytic.

We have to show that $\dim \ker R^*=\rank (I-P)TP<\infty$. Note that $C= (I-P)TP$ and clearly this has finite rank, since $P$ has finite rank. Furthermore, by construction $\rank C= \dim \ker S^*$. We will finish the proof by showing that $\dim \ker R^*=\dim \ker S^*$. Let $z\in \D\setminus \sigma(A)$. Then $\dim \ker R^*=\dim \ker (R^*-\overline{z})$, $\dim \ker S^*=\dim \ker (S^*-\overline{z})$, and
$$\ker ( R^*-\overline{z})=\{x\oplus -(A^*-\overline{z})^{-1}C^*x : x\in \ker(S^*-\overline{z})\}.$$ The theorem follows.
\end{proof}

\section{De Branges-Rovnyak spaces and expansive operators}\label{Operator deBranges}
Let $\HE$ and $\HD$ be  complex Hilbert spaces, let $H^2(\HD)$ be the space of square summable power series with coefficients in $\HD$ and let $\HS(\HD, \HE)$ be the Schur class functions, i.e. the collection of functions that are analytic on $\D$ take values in the contractive operators in $\HB(\HD,\HE)$. A function $K:\D \times \D \to \HB(\HE)$ is called a kernel, if it is positive definite in the  sense that whenever $n\in \N$ and  $x_1, \dots, x_n \in \HE$, $z_1, \dots , z_n\in \D$, then $\sum_{i,j=1}^n \la K(z_i,z_j)x_i,x_j\ra \ge 0$. It is well-known that if $K$ is a kernel, then there is a unique Hilbert space $\HH$ of functions on $\D$ with values in $\HE$ such that for each $z\in \D$ and $x\in \HE$ the function $K_zx=K(\cdot,z)x\in \HH$ and $\la f(z),x\ra_{\HE}=\la f, K_z x\ra_{\HH}$ for all $f\in \HH$.

It is also well-known that if $B\in \HS(\HD,\HE)$ is a Schur function, then $$K^B_w(z)= \frac{I_{\HE}-B(z)B(w)^*}{1-z\overline{w}}$$ is a kernel. The corresponding reproducing kernel Hilbert space is the de Branges-Rovnyak space $\HH(B)$. The backward shift $Lf(z)=\frac{f(z)-f(0)}{z}$ acts contractively on each de Branges-Rovnyak space, and $\HH(B)$ is contractively contained in $H^2(\HE)$. Note that if $B(0)=0$, then $K^B_0(z)=I_{\HE}$, hence the constant functions are contained in $\HH(B)$, $\HE\subseteq \HH(B)$, and in fact $\ker L=\HE$. We will need a few further facts about the connections between $L$, $B$, and $K_w^B(z)$ in the case when $B(0)=0$.

\begin{lemma} \label{1.1} Let $L_0\in \HB(\HH)$ be a Hilbert space contraction, let $D_*=(I-L_0L_0^*)^{1/2}$, $\HD_*=\overline{\ran D_*}$, and let $P$ denote the projection onto $\HE_*=\ker L_0$.

Then $$B(z)=zP(1-zL_0)^{-1}D_*|\HD_*$$ is in the Schur class $\HS(\HD_*, \HE_*)$. $B$ satisfies $B(0)=0$ and
$$K^B_w(z)= P(1-zL_0)^{-1}(1-\overline{w}L_0^*)^{-1}P.$$

Furthermore, the map $V$, $$Vx(z)=P(I-zL_0)^{-1}x$$ defines a partial isometry of $\HH$ onto $\HH(B)$. It satisfies $VL_0=LV$ and $VD_*x(z)=\frac{B(z)}{z}x$ for all $x\in \HH$ and $z\in \D$.
\end{lemma}
\begin{proof} It is clear that $B(0)=0$ and the fact that $B$ is in the Schur class will follow once we establish the identity
$K^B_w(z)= P(1-zL_0)^{-1}(1-\overline{w}L_0^*)^{-1}P$ since the expression on the right is clearly a kernel. However, for the bigger picture it is useful  to note that  the colligation operator $U:\HE_* \oplus \HH\to \HD_* \oplus \HH$ defined by $$U=\left[\begin{matrix} 0 &D_*\\P&L_0^*\end{matrix}\right]$$ is isometric, and hence $$B(z)^*= \overline{z} D_*(I-\overline{z}L_0^*)^{-1}P$$ is the transfer function realization of the Schur function $B\in \HS(\HD_*, \HE_*)$, see e.g. \cite{AM02}, Corollary 8.26.

 Write $K_w(z)=P(1-zL_0)^{-1}(1-\overline{w}L_0^*)^{-1}P$, and for later reference  note that   $PL_0^*=L_0P=0$ implies that
\begin{align}\label{kernel} K_w(z)=P+P(I-zL_0)^{-1}(zL_0)(\overline{w}L_0^*)(I-\overline{w}L_0^*)^{-1}P.\end{align}

 Now calculate
\begin{align*}B(z)B(w)^*&=z\overline{w} P(I-zL_0)^{-1}D_*^2(I-\overline{w}L_0^*)^{-1}P\\
&= z\overline{w} P(I-zL_0)^{-1}(I-L_0L_0^*)(I-\overline{w}L_0^*)^{-1}P\\
&= z\overline{w} K_w(z)-\left( P(I-zL_0)^{-1}(zL_0)(\overline{w}L_0^*)(I-\overline{w}L_0^*)^{-1}P\right)\\
&= z\overline{w} K_w(z)-\left(K_w(z)-P\right) \ \text{by (\ref{kernel})}\\
&=P -(1-z\overline{w} )K_w(z)
\end{align*}
 Noting that $P$ acts as the identity on $\mathcal E_*$ we obtain $K_w(z) = \frac{I-B(z)B(w)^*}{1-z\overline{w}}=K^B_w(z).$

The linear function $V$ defined on $\HH$ by $Vx(z)=P(I-zL_0)^{-1}x$ takes values in $\Hol(\D,\HE)$, the analytic $\HE$-valued functions on $\D$. The kernel of $V$ equals $\bigcap_{n\ge 0} \ker PL_0^n$. In order to show that $V$ is a partial isometry into $\HH(B)$, it suffices to show that $V$ is isometric on the set
$$\left(\bigcap_{n\ge 0} \ker PL_0^n\right)^\perp = \mathrm{ Span}_{n\ge 0} \ran(L_0^*)^nP=  \mathrm{ Span}_{|w|<1} \ran(I-\overline{w}L_0^*)^{-1}P.$$
But if $x= \sum_{i=1}^n (I-\overline{w_i}L_0^*)^{-1}Px_i$ for some
$w_1, \dots, w_n \in \D$ and $x_1,\dots x_n \in \HH$,
then \begin{align}\label{Vkernel}Vx(z) =\sum_{i=1}^n K^B_{w_i}(z)x_i \in \HH(B),\end{align}
and
\begin{align*}\|Vx\|^2_{\HH(B)}&=\sum_{i,j=1}^n \la K^B_{w_i}(w_j)x_i,x_j\ra\\
&=\sum_{i,j=1}^n \la P(1-w_iL_0)^{-1}(I-\overline{w}_jL_0^*)^{-1}Px_i,x_j\ra\\
&=\left\|\sum_{i=1}^n (I-\overline{w_i}L_0^*)^{-1}Px_i\right\|^2\\
&=\|x\|^2.
\end{align*}
Thus, $V$ defines a partial isometry and it must be onto by (\ref{Vkernel}). The identities $VL_0=LV$ and $VD_*x(z)=\frac{B(z)}{z}x$ for all $x\in \HH$ and $z\in \D$ follow easily.
\end{proof}

\begin{remark} A main result of de Branges and Rovnyak states that every Hilbert space contraction $X$ that is completely non-isometric, i.e.
$$\{x\in \HH: \|X^nx\|=\|x\| \text{ for all } n \ge 0\}=\{0\},$$
is unitarily equivalent to a backward shift acting on some $\HH(B)$ space (and conversely).\end{remark} That result can be seen to follow from our construction as follows.  Given a completely non-isometric contraction $X$, let $D=(I-X^*X)^{1/2}$, $\HE=\overline{\ran D}$, then the operator
$$L_0= \left[\begin{matrix} X &0\\D&0\end{matrix}\right] \text{ acting on } \HH\oplus \HE $$
is a contraction, and if we apply the construction of the previous lemma to $L_0$, then the resulting $V:\HH\oplus \HE\to \HH(B)$ is 1-1, hence unitary. Furthermore, one easily checks that $X$ is unitarily equivalent to $(L, \HH(B_1))$, where $B_1(z)=B(z)/z$. We omit the details.

 For the forward shift $Y: f\to zf$ to act boundedly on $\HH(B)$ an extra condition is required. For scalar Schur functions $b$ the condition is that $b$ is not an extreme point of the unit ball of $H^\infty$, or equivalently $\log(1-|b|)\in L^1(\T)$. The condition turns out to be equivalent to the condition that $b\in \HH(b)$.  For operator-valued Schur functions the following lemma describes the analogue of this condition for the case where $B(0)=0$.
\begin{lemma}\label{condfsfwsb}
Let $B\in \HS(\HD_*,\HE)$ with $B(0)=0$,  then  the following are equivalent:
\begin{enumerate}
\item $Y: f\to zf$ defines a bounded operator on $\HH(B)$,
\item for each $x\in \HD_*$ the function $g_x(z)=B(z)x\in \HH(B)$.
\end{enumerate}
Indeed, if $x\to \tau(x)$ is defined by $\tau(x)(z)=B(z)x$, then $Y$ is bounded on $\HH(B)$, if and only if $\tau:\HD_*\to \HH(B)$ is bounded with $\|\tau\|^2=\|Y\|^2-1$.
\end{lemma}

\begin{proof}
Let $0<c<1$, then one verifies the identity
\begin{align*}\frac{I-B(z)B(w)^*}{1-z\overline{w}}&-\frac{c}{1-c}B(z)B(w)^*-I\\
&=\frac{1}{1-c}\left(\frac{(1-cz\overline{w})(I-B(z)B(w)^*)}{1-z\overline{w}}-I\right).\end{align*}

Note that if $K_w^B(0)=I$, then $K_w^B(z)$ is positive definite, if and only if $K_w^B(z)-I$ is positive definite. Hence it follows that
\begin{align}\label{positiveone}
\frac{(1-cz\overline{w})(I-B(z)B(w)^*)}{1-z\overline{w}}
\end{align}
is positive definite, if and only if
\begin{align}\label{positivetwo}
\frac{I-B(z)B(w)^*}{1-z\overline{w}}-\frac{c}{1-c}B(z)B(w)^*
\end{align}
is positive definite. The first condition (\ref{positiveone}) is equivalent to $Y$ being bounded with $\|Y\|^2\le 1/c$, while the second condition (\ref{positivetwo}) is equivalent to the map $x \to \tau(x)$, $\tau(x)(z)=B(z)x$ being bounded with $\|\tau\|^2\le \frac{1}{c}-1$. Indeed, for (\ref{positivetwo}) one checks that the positivity of $\frac{I-B(z)B(w)^*}{1-z\overline{w}}-\frac{c}{1-c}B(z)B(w)^*$ is equivalent to the inequality $$\|\sum_{i}B(z_i)^*x_i\|^2\le (\frac{1}{c}-1) \|\sum_{i}K_{z_i}^B(\cdot)x_i\|^2,$$ which gives the norm inequality for $\tau^*$.
\end{proof}

\begin{example}\label{DiriExample} The Dirichlet space $D$ has reproducing kernel
$$k_w(z)=\frac{1}{\overline{w}z}\log \frac{1}{1-\overline{w}z}=\sum_{n=0}^\infty \frac{\overline{w}^nz^n}{n+1}.$$
In this case one can take $B(z)=\{b_n\}_{n\ge 1}$, $b_n(z)= \frac{z^n}{\sqrt{n(n+1)}}$. Then for $x=\{a_n\}\in \ell^2$ we have $B(z)x=\sum_{n\ge 1} a_nb_n(z)=\sum_{n\ge 1} \frac{a_nz^n}{\sqrt{n(n+1)}}\in D$. Note that $\|B(e^{it})\|=1$ for each $t$. Thus, in this case $M_z$ is bounded although $\log(1-\|B(e^{it})\|) \notin L^1(\T)$.
\end{example}

Let $B$ be any Schur function. Then since $X$ is a contraction, if $M_z$ acts boundedly on $\HH(B)$, then $\|M_zf\|\ge \|XM_zf\|=\|f\|$ for all $f\in \HH(B)$, i.e. $M_z$ is norm expansive. Let $T\in \HB(\HH)$ be such that $\|Tx\|\ge \|x\|$ for all $x\in \HH$. Then $\ran T^n$ is closed for each $n$, and recall that $T$ is analytic, if $\bigcap_{n\ge 0} \ran T^n=(0)$. Thus, whenever $M_z$ is bounded on $\HH(B)$, then as $\HH(B) \subseteq H^2(\HE)$ it is obvious that $M_z$ is analytic. The next theorem says that all norm expansive analytic operators can be modelled as $M_z$ on some $\HH(B)$.
\begin{lemma}\label{rangeofDelta and D*} If $T\in \HB(\HH)$ satisfies $\Delta=T^*T-I\ge 0$, and if $L=(T^*T)^{-1}T^*$ is the left inverse of $T$ with $\ker L=\ker T^*$, then $$\overline{\ran \Delta} =\overline{ \ran D_*}, \ \ D_*=(I-LL^*)^{1/2}.$$
\end{lemma}
\begin{proof} The lemma follows from the identity $D_*^2=I-(T^*T)^{-1}= I-(I+\Delta)^{-1}= \Delta (I+\Delta)^{-1}$.
\end{proof}
\begin{theorem}\label{expandingOp} Let $T\in \HB(\HH)$. Then the following are equivalent
\begin{enumerate}
\item $T$ is analytic and norm expanding,
\item there are Hilbert spaces $\HE$ and $\HD$, a Schur function $B\in \HS(\HD, \HE)$ such that $B(0)=0$,  $M_z\in \HB(\HH(B))$, and $T$ is unitarily equivalent to $(M_z,\HH(B))$,
\item there are Hilbert spaces $\HE$ and $\HD$, a Schur function $B\in \HS(\HD, \HE)$,  $M_z\in \HB(\HH(B))$, and $T$ is unitarily equivalent to $(M_z,\HH(B))$.
\end{enumerate}
If the conditions are satisfied, then in (ii) one can take $\HE= \ker T^*$ and $\HD=\overline{\ran (T^*T-I)}$.
\end{theorem}

\begin{proof}
 (ii) $\Rightarrow$ (iii) is trivial and we already noted that (iii) $\Rightarrow$ (i).

(i) $\Rightarrow$ (ii):  Let $T$ be analytic and norm expanding. Then $T^*T$ is invertible and $L=(T^*T)^{-1}T^*$ is a left inverse of $T$ with $\HE=\ker L=\ker T^*$. The operator $L^*$ has been called the Cauchy dual of $T$, see \cite{Sh01}.  Note that  $I-TL=I-L^*T^*=P$, the projection onto $\HE$.  Since $\HH= \ran T \oplus \ker T^*$ one easily sees that $\|L\|\le 1$. Thus, as in Lemma \ref{1.1} for any $x\in \HH$ one can define the $\HE$-valued holomorphic function $Vx:\D \to \HE$ by $Vx(z)= P(I-zL)^{-1}x$, $z\in \D$. In \cite{Sh01}, Lemma 2.2, Shimorin showed that the analyticity of $T$ implies that $V$ is 1-1, and hence by Lemma \ref{1.1} $V:\HH \to \HH(B)$ is unitary and one easily verifies that $VT=M_zV$.

By Lemma \ref{1.1} we can take $B\in \HS(\overline{\ran D_*}, \ker L)$. We already noted that $\ker T^*=\ker L$ and by Lemma \ref{rangeofDelta and D*} we have $\overline{\ran (1-LL^*)^{1/2}}=\overline{\ran(T^*T-I)}$.
\end{proof}

Theorem \ref{expandingOp} implies that for a study of operators of the type $(M_z,\HH(B))$ there is no loss in generality, if we assume that $B(0)=0$. In the remainder of this paper that will be the standard assumption. Of course, that means that all hypotheses and conclusions about $B$ and the mate $a$ are only valid if $B(0)=0$. Thus, it will be useful to know how to pass to the general case. We will now explain how to do this for scalar-valued Schur functions $B$. In fact, it turns out that for two contractive analytic functions  $B=(b_1,b_2,\dots)$ and $C=(c_1,c_2, \dots)$ the operators  $(M_z,\HH(B))$ and $(M_z,\HH(C))$ are unitarily equivalent to one another, if and only if $B$ and $C$ are related by a ball automorphism.

For $k\in \N\cup\{\infty\}$ let $\B_k$ be the unit ball in $\C^k$ (or $\ell_2$ if $k=\infty$), and for $\alpha \in \B_k$ let $\varphi_\alpha$ be the analytic  automorphism $\B_k\to \B_k$ defined by $\varphi_\alpha(z)=-z$, if $\alpha=0$ and
$$\varphi_\alpha(z)= \frac{\alpha-P_\alpha z - (1-\|\alpha\|^2)^{1/2} Q_\alpha z}{1-\la z, \alpha \ra},$$ if $\alpha \in \B_k\setminus\{0\}$. Here $P_\alpha z=\frac{\la z,\alpha\ra}{|\alpha|^2} \alpha$, $Q_\alpha=I_{\C^k}-P_\alpha$.
It satisfies $\varphi_\alpha(\alpha)=0$, $\varphi_\alpha(0)=\alpha$,  $\varphi_\alpha^{-1}=\varphi_\alpha$, and
$$1-\la \varphi_\alpha(z), \varphi_\alpha(w)\ra= \frac{(1-\|\alpha\|^2)(1-\la z, w\ra)}{(1-\la z,\alpha\ra)(1-\la \alpha,w\ra)},$$ see \cite{Ru08}.

Thus, if $B=(b_1,\dots, b_k)$ or $B=(b_1,b_2, \dots)$ is a $\B_k$-valued  analytic function on $\D$, then
$$\frac{1-\la \varphi_\alpha(B(z)), \varphi_\alpha(B(w))\ra}{1-z\overline{w}}=\frac{(1-\|\alpha\|^2)}{(1-\la z,\alpha\ra)(1-\la \alpha,w\ra)}\frac{1-\la B(z), B(w)\ra}{1-z\overline{w}}.$$ This implies that if we write $B_\alpha(z)=\varphi_\alpha(B(z))$ and $f_\alpha(z)=\frac{\sqrt{1-\|\alpha\|^2}}{1-\la z, \alpha\ra}$, then $$K^{B_\alpha}_w(z)= f_\alpha(z)\overline{f_\alpha(w)}K^B_w(z).$$ Thus, the following lemma is obvious.
\begin{lemma} \label{unitaryBalpha} For each $\alpha \in \B_k$ and each $\B_k$-valued analytic function $B$ the operator $M_{f_\alpha}: \HH(K^B)\to \HH(K^{B_\alpha}), g\to f_\alpha g$ is unitary.

Furthermore, $T=(M_z,\HH(B))$ is bounded, if and only if $T_\alpha=(M_z,\HH(B_\alpha))$ is bounded and $T_\alpha M_{f_\alpha}=M_{f_\alpha}  T$.
\end{lemma}
We conclude that if a $\B_k$-valued analytic function $B$ is given, and if $T=(M_z,\HH(B))$ is  a bounded operator, then by taking $\alpha=B(0)$ we obtain a $B_\alpha$ with $B_\alpha(0)=0$ and $T$ is unitarily equivalent to $(M_z,\HH(B_\alpha))$.

The next lemma can be considered to be a converse of the above observation.
\begin{lemma} \label{unitaryEquiv} Let $k, j\in \N\cup \{\infty\}$,  let $B$ be $\B_k$-valued analytic function with $B(0)=0$, and let $C$ be a $\B_j$-valued analytic function with $C(0)=0$.

If  $T=(M_z,\HH(B))$ and $S=(M_z,\HH(C))$ are bounded and unitarily equivalent to one another, then $K^C_w(z)=K^B_w(z)$ for all $z,w\in \D$ and  there is a partial isometry $V:\C^k\to \C^j$ (or $\ell_2\to \ell_2$, etc.) such that $C(z)=V(B(z))$ for each $z\in \D$.
\end{lemma}
\begin{proof} Suppose $T=(M_z,\HH(B))$ is a bounded operator. We claim that for each $\lambda\in \D$ we have $\ker (T-\lambda)^*=\{\gamma K^B_\lambda: \gamma \in \C\} $. It is clear that $K^B_\lambda\in \ker (T-\lambda)^*$, thus it suffices to show that  $\dim \ker (T-\lambda)^*=1$ for each $\lambda \in \D$.

If $f(0)=0$, then $\frac{f(z)}{z}=\frac{f(z)-f(0)}{z}= Lf(z)\in \HH(B)$ and hence $f \in \ran T$. Thus, $\dim \ker T^*=1$ and as $T$ is expansive, the Fredholm theory implies that $\dim \ker (T-\lambda)^*=\dim \ker T^*=1$ for each $\lambda \in \D$. Similarly, we conclude that $\ker (S-\lambda)^*=\{\gamma K^C_\lambda: \gamma \in \C\} $.

Let $U: \HH(B)\to \HH(C)$ be unitary such that $UT=SU$. Then $U^*(\ker (S-\lambda)^*)=\ker (T-\lambda)^*$ for each $\lambda \in \D$. This implies that $U^*K_\lambda^C=\overline{f(\lambda)}K^B_\lambda$ for some value $f(\lambda)\in \C$. This leads to the equality
$$\frac{1-\la C(z),C(w)\ra}{1-z\overline{w}}=f(z) \overline{f(w)}\ \frac{1-\la B(z),B(w)\ra}{1-z\overline{w}}$$ for all $z,w\in \D$.
Taking  $z=w=0$ we see that the hypothesis that $C(0)=B(0)=0$ implies that $f(0)$ has modulus 1. Similarly,  we take $w=0$, and conclude that for all $z\in \D$ the identity $1=f(z)\overline{f(0)}$ holds, so $f$ must be constant. This implies that $\la C(z), C(w)\ra=\la B(z),B(w)\ra$ for all $z,w\in \D$. Thus $K^C=K^B$.

Let $\HM=\bigvee\{B(z):z\in \D\}\subseteq \C^k$ (or $\ell_2$), and define a linear transformation $V:\HM\to \C^j$(resp. $V:\HM\to \ell_2$) by $V(B(z))=C(z)$. The identity $\la C(z), C(w)\ra=\la B(z),B(w)\ra$ implies that $V$ is well-defined and isometric. It becomes a partial isometry on $\C^k$, if we set it equal to 0 on $\HM^\perp$. This concludes the proof of the lemma. \end{proof}
\section{Some observations about Schur functions with  $B(0)=0$.}
Note that if $B\in \HS(\HD,\HE)$ is such that $B(0)=0$, then $K_0^B(z)=I_\HE$ and this implies that the constant functions form the orthocomplement in $\HH(B)$ of the functions that are 0 at 0. Obviously the constant functions form the null space of $L$, the backward shift on $\HH(B)$. If $T=(M_z,\HH(B))$ is bounded, then this implies that $\ker T^*=\ker L$. Since the operator $(T^*T)^{-1}T^*$ is a left inverse of $T$ with null space equal to $\ker T^*$, we conclude that the hypothesis $B(0)=0$ implies that $L=(T^*T)^{-1}T^*$, i.e. the backward shift equals the operator from the construction of the proof of Theorem \ref{expandingOp}.

\begin{lemma} \label{rangeofDelta} If $B\in \HS(\HD,\HE)$ with $B(0)=0$ is such that $T=(M_z,\HH(B))$ defines a bounded operator on $\HH(B)$, then
$$\overline{\ran \Delta}= \clos\{g_x: x\in \HD, g_x(z)=\frac{B(z)}{z}x\}, \ \Delta=T^*T-I.$$
\end{lemma}
\begin{proof} As above let $L$ be the backward shift on $\HH(B)$, and $D_*=(I-LL^*)^{1/2}$. Then by Lemma \ref{rangeofDelta and D*} we have $\overline{\ran \Delta}=\overline{\ran D_*}$.

If $y\in \HE$ and $ w\in \D$, then set $f= K^B_w(\cdot)y$. One calculates that $((I-LL^*)f)(z)= \frac{B(z)}{z} \frac{B(w)^*}{\overline{w}}y$. Since elements of the type as $f$ span $\HH(B)$ we have
\begin{align} \label{equranD} \overline{\ran D_*}&=\overline{\ran D_*^2}\\
&= \bigvee_{w\in \D}\{g_{w,y}: g_{w,y}(z)=\frac{B(z)}{z} \frac{B(w)^*}{\overline{w}}y, y\in \HE\} \notag\\
&=\clos \{g_x:x\in \HD, g_x(z)=\frac{B(z)}{z}x\}, \notag\end{align}
where the last equality follows, because the function $\frac{B(z)}{z}a$ is identically equal to 0, if $a\perp \ran \frac{B(w)^*}{\overline{w}}$ for all $w\in \D$.
\end{proof}

\begin{lemma} \label{reducing subspace} Let $B\in \HS(\HD,\HE)$ be such that $B(0)=0$ and such that  $T=(M_z,\HH(B))$ is bounded.

If $\HM$ is a reducing subspace for $T$, and if $\HE_0=\HM\cap \ker T^*$, then $f(z)\in \HE_0$ for every $f\in \HM$ and $z\in \D$, and $\HE_0$ is reducing for $K^B_w(z)$ and for $B(z)B(w)^*$ for every $z,w\in \D$.
 \end{lemma}
\begin{proof} As before the backward shift $L$ satisfies $L=(T^*T)^{-1}T^*$ and $L^*=T(T^*T)^{-1}$. Thus the hypothesis that $\HM$ is reducing for $T$ implies that it also reduces $L$. Hence, if $f\in \HM$, then $f(0)=f-TLf\in \HM\cap \ker T^*=\HE_0$. If for such $f$ we have $f(z)=\sum_{n=0}^\infty a_nz^n$, then for each $n$,  $a_n=L^nf(0)\in \HE_0$. This implies $f(z)\in \HE_0$ for all $z\in \D$.

By symmetry we have $f(w)\in \HE\ominus \HE_0$ for all $f\in \HM^\perp$ and $w\in \D$. Let $x\in \HE_0$, then  for every $f\in \HM^\perp$ and $w\in \D$ we have
$$0=\la f(w), x\ra_\HE= \la f, K_w^B(\cdot) x\ra_{\HH(B)}.$$ This implies that $K^B_w(\cdot)x\in \HM$ and hence by the first part of the proof $K^B_w(z)x\in \HE_0$ for all $z, w\in \D$. Thus, $\HE_0$ is invariant for $K^B_w(z)$, and by symmetry it must be reducing. Next note that $B(z)B(w)^*= I-(1-z\overline{w})K^B_w(z)$. This implies that $\HE_0$ is also reducing for $B(z)B(w)^*$.
\end{proof}

\begin{lemma} \label{isometric summand} Let $B\in \HS(\HD,\HE)$ be such that $B(0)=0$ and such that  $T=(M_z,\HH(B))$ is bounded.

Then  $\HE=\bigvee_{z\in \D} \ran B(z)$, if and only if $(M_z,\HH(B))$ has no nontrivial reducing subspace on which it acts isometrically.
\end{lemma}
\begin{proof} Let $\HR_B=\bigvee_{z\in \D} \ran B(z)$ and suppose that $\HR_B \ne \HE$. Then
$$K^B_w(z)=\frac{1}{1-z\overline{w}}I_{\HR_B^\perp} \oplus K^{B_1}_w(z),$$ where $B_1\in \HS(\HD,\HR_B)$ is defined by $B_1(z)=P_{\HR_B}B(z)$. It is then clear that $\HH(B)=H^2(\HR_B^\perp) \oplus \HH(B_1)$ and  $(M_z,\HH(B))$ has a unilateral shift of multiplicity $\dim \HR_B^\perp$ as a direct summand.

Now suppose $T$ has a nontrivial reducing subspace on which it acts isometrically. Then by Lemma \ref{reducing subspace} there is a nontrivial subspace $\HE_0$ of $\HE$ that reduces $B(z)B(w)^*$  for all $z,w\in \D$. Furthermore, since $T_0=(M_z,\HH(B_0))$ is isometric, if and only if $B_0=0$, we must have $B(z)B(w)^*|\HE_0=0$ for all $z,w\in \D$. But then $\ran B(z)B(w)^*\subseteq \HE\ominus \HE_0$ and this implies $\HE_0 \perp \ran B(z)$ for all $z\in \D$.
\end{proof}
\section{Rational matrix-valued Schur class functions}

\begin{lemma} \label{T*Linvariant} If $T\in \HB(\HH)$ satisfies $\Delta=T^*T-I\ge 0$, and if $L=(T^*T)^{-1}T^*$ is  the left inverse of $T$ with $\ker L=\ker T^*$, then
$$[\ran \Delta]_{T^*}=[\ran \Delta ]_{L}.$$
\end{lemma}
\begin{proof} We start by showing that $[\ran \Delta]_{T^*}$ is $L$-invariant. That will imply one of the inclusions and the other one will follow analogously. Let $x\in [\ran \Delta]_{T^*}$, then set $y= (I+\Delta)^{-1}T^*x$. Since $L= (I+\Delta)^{-1}T^*=T^* - \Delta(I+\Delta)^{-1}T^*$ we have $Lx=T^*x + \Delta y\in [\ran \Delta]_{T^*}$.

Similarly, if $x\in [\ran \Delta ]_{L}$, then $T^*x=(I+\Delta)Lx=Lx + \Delta Lx\in [\ran \Delta ]_{L}$.
\end{proof}

If $\HD$ and $\HE$ are arbitrary Hilbert spaces and $B\in \HS(\HD,\HE)$, then we will say that $B$ is rational, if there is a scalar polynomial $q\ne 0$ and an operator polynomial $P$ such that $B(z)=\frac{1}{q(z)} \ P(z)$. Recall that the degree of $B$ is defined to be the smallest integer $n$ such that there are such polynomials both of which have degree $\le n$.

\begin{theorem} \label{ranDelta and rational} Let  $B\in \HS(\HD,\HE)$ with $B(0)=0$ be such that $T=(M_z,\HH(B))$ defines a bounded operator on $\HH(B)$. Write $\Delta=T^*T-I$, $\HN=[\ran \Delta]_{T^*}$ and $\HR_B = \bigvee_{z\in \D} \ran B(z)$.

 Then $\dim \HN < \infty$, if and only if $B$ is rational and $\dim \HR_B<\infty$. In fact, $$\text{degree B} \le \dim \HN \le \dim \HR_B \ \text{degree B}.$$

Furthermore, if $\dim \HN=n$, then there is an operator-valued polynomial $P(z)=\sum_{k=1}^n P_k z^k$, $P_k \in \HB(\HD,\HE)$ such that
$B(z) = \frac{1}{\tilde{q}(z)}P(z)$, where $q$ is the characteristic polynomial of $L|\HN$, and $\tilde{q}(z)=z^n q(1/z)$.
\end{theorem}
\begin{proof} If $\HR_B\ne \HE$, then by Lemmas \ref{reducing subspace} and \ref{isometric summand} $T$ will have a nontrivial reducing subspace $\HM$  on which it acts isometrically. This subspace will be reducing for $L$ and $\Delta$, and $\Delta|\HM=0$ and $\HN\subseteq \HM^\perp$. Thus, there will be no loss in generality, if we assume that $\HE=\HR_B$.

Suppose $\HN=[\ran \Delta]_{T^*}$ has dimension $n$. Then by Lemma \ref{T*Linvariant} $\HN$ is invariant for $L$, hence $L|\HN$ can be represented by an $n\times n$ matrix. Let $q(z)=\sum_{j=0}^n \hat{q}_jz^j$ be the characteristic polynomial of this matrix. Let $x\in \HD$, then by Lemma \ref{rangeofDelta} the function $g(z)=\frac{B(z)}{z}x\in \overline{\ran \Delta}$ and hence $q(L)g=0$. Let $B(z)=\sum_{n=1}^\infty B_n z^n$ for $B_n\in \HB(\HD,\HE)$. Then for each $j\in \N$ we have $z^j(L^jg)(z)= \frac{B(z)}{z}x-\sum_{k=1}^j (B_{k}x)z^{k-1}$. Hence
\begin{align*}0=z^n (q(L)g)(z)&= \sum_{j=0}^n \hat{q}_j z^{n-j} z^j(L^jg)(z)\\
&=z^nq(1/z)\frac{B(z)}{z}x -\sum_{j=0}^n \hat{q}_j z^{n-j}\sum_{k=1}^j (B_{k}x) z^{k-1}.\end{align*}
Thus, $$\tilde{q}(z)\frac{B(z)}{z}x=\sum_{j=0}^n \sum_{k=1}^j \hat{q}_j (B_{k}x) z^{n+k-j-1},$$
which is a polynomial of degree $\le n-1$, hence
$B(z)= \frac{1}{\tilde{q}(z)} P(z)$, where $P$ is an operator polynomial of degree $\le n$. Thus, $\text{ degree }B \le n$.

Next we show that $\dim \HR_B <\infty$. By Lemma \ref{isometric summand} the assumption that $\HR_B=\HE$ implies that $T$ does not have an isometric direct summand, hence by Theorem \ref{Wold-type} we have $\dim \ker T^*= \rank (I-P_\HN)TP_\HN$, where $P_\HN$ denotes the finite rank projection onto $\HN$. Since $\HR_B=\ker T^*$ we conclude that $\HR_B$ is finite dimensional.

Now suppose $\dim \HR_B <\infty$ and $B$ is rational with $\text{ degree B}=n<\infty$, i.e.  there is a scalar polynomial $q$ of degree $\le n$ and an operator polynomial $P$ of degree $\le n$ such that $P(z)=\tilde{q}(z)B(z)$. Then $P(0)=0$.
  Let $P_0(z)=P(z)/z$, $B_0(z)=B(z)/z$. Then for all $x\in \HD$ we have $$q(L)B_0 x=L^n P_0x=0.$$ Thus by Lemma \ref{rangeofDelta} $\ran \Delta \subseteq \ker q(L)$, and hence $\HN\subseteq \ker q(L)$. Since $q$ has degree $\le n$, there is a $k \le n$ and $\lambda_1, \dots, \lambda_k\subseteq \D$ such that $q(z)= q_1(z)\prod_{j=1}^k(z-\lambda_j)$, where $q_1$ is a polynomial without zeros in $\D$. Since $\|L^mh\|_{H^2(\HE)} \to 0$ for every $h\in H^2(\HE)$, $L$ cannot have any eigenvalues in $\C\setminus \D$, hence  $\ker q_1(L)=(0)$. Furthermore, since $M_z$ does not have any eigenvalues the operator $(I-\lambda_jM_z)$ is 1-1 for each $j$, thus the identity
$\prod_{j=1}^k(L-\lambda_j)= L^k\prod_{j=1}^k(I-\lambda_jM_z)$ shows that \begin{align*}\dim \HN &\le \dim \ker q(L)\le \dim \ker L^k \\
&=k \dim \ker L\le n \dim \HE= n \dim \HR_B.\end{align*}
\end{proof}

Now Theorem \ref{representationVR} follows form Theorems \ref{Wold-type}, \ref{expandingOp} and \ref{ranDelta and rational}.

\section{Rational Row Schur functions}
We will now restrict attention to the situation where $K^B_w(z)$ is a reproducing kernel for a nonzero space of scalar-valued functions, $\HH(B)\subseteq H^2$. Then $u_w(z)=B(z)B(w)^*$ is scalar-valued and it is the reproducing kernel for a Hilbert space of analytic functions $\HH(u)$. If $\{b_i\}$ is any orthonormal basis of $\HH(u)$, then $B(z)B(w)^*=\sum_{i\ge 1}b_i(z)\overline{b_i(w)}$. If we assume that $T=(M_z,\HH(B))$ is bounded, then the dimension of $\HH(u)$ equals the rank of $\Delta=T^*T-I$, see equation (\ref{equranD}) of the proof of Lemma \ref{rangeofDelta}. As we will be interested in the situation where $\Delta$ has finite rank, we can assume that $B=(b_1,\dots, b_k)$ for some linearly independent set of functions $b_1,\dots, b_k$. By the theorem of Aleman and Malman (Theorem \ref{AleMalTheorem}) in order to assure that $(M_z,\HH(B))$ be bounded we will need to assume that $1-\sum_{i=1}^k|b_i|^2$ is log-integrable. Then there will be a unique outer function $a$ such that $a(0)>0$ and  $|a|^2+\sum_{i=1}^k|b_i|^2=1$ a.e. on the unit circle. We say that $a$ is the mate of $B$.

Now assume that $B$ is rational of degree $ n$, then $a$ will be rational of degree $ \le n$. Indeed, if there are polynomials $q, p_1, \dots, p_k$ of degree $\le n$ such that $b_i=p_i/q$ and if $1-\sum_{i=1}^k|b_i|^2$ is log-integrable, then $|q|^2-\sum_{i=1}^k|p_i|^2 >0$ a.e. on $\T$, and hence by the Fej\'{e}r-Riesz Theorem \cite{RN90}, there is a polynomial $p$ of degree $\le n$ with no zeros in $\D$ and such that $|q|^2-\sum_{i=1}^k|p_i|^2=|p|^2$ on the unit circle. Then $a=p/q$ is the required outer function. We note that the denominator of $a$ can be chosen to be the same as the common denominator of the $b_i's$. Theorem \ref{Wold-type} implies that under this hypothesis we have $\dim \ker T^*=1$, and Theorem \ref{ranDelta and rational} tells us that $$\HN=[\ran \Delta]_{T^*}=[\ran \Delta]_{L}  \text{ has dimension }n<\infty.$$

The generalized eigenspaces of $L$ are all of the form $\ker (L-\lambda)^n=\{\sum_{j=1}^n \frac{a_j}{(1-\lambda z)^j}: a_j \in \C\}$, $0<|\lambda|<1$, and $\ker L^n=\{\sum_{j=0}^{n-1} a_jz^j:a_j \in \C\}$. That is well-known to be true for the backward shift on $H^2$, hence it now follows from the fact that all the functions in the generalized eigenspaces are contained in $\HH(B)\subseteq H^2$. That means that there are $\alpha_1, \dots, \alpha_n \in \D$ such that $\HN$ is of the form
\begin{align}\label{N as rational functions}\HN=\{\frac{p(z)}{\prod_{j=1}^n(1-\alpha_j z)}: p \text{ is a polynomial of degree }<n\}.\end{align}

 In fact, in Theorem \ref{ranDelta and rational} we saw that if $B(0)=0$, then ${\prod_{j=1}^n(1-\alpha_j z)}$ is a constant multiple of the lowest common denominator of the $b_i$'s.

\begin{lemma} \label{eigenspaces of M^*} Let $B=(b_1,\dots,b_k)$ be rational functions such that $\log(1-\sum_i|b_i|^2)\in L^1(\T)$ and $B(0)=0$, but $B\ne 0$, then all eigenspaces of $M_z^*$  are one dimensional.

This implies that the minimal and characteristic polynomials of $M_z^*|\HN$ coincide.
\end{lemma}
\begin{proof} By a theorem of Aleman and Malman (\cite{AlemanMalman}, Theorem 5.5) the polynomials are dense in $\HH(B)$, thus the operator $M_z$ has a cyclic vector. The lemma follows, because it is well-known that if $T$ is any Hilbert space operator with a cyclic vector $x_0$, then every eigenspace is one dimensional. Indeed, for $w\in \C$ we have that the set $\{x_0, (T-w)x_0, (T-w)^2x_0, \dots\}$ spans $\HH$, and hence $\overline{\ran (T-w)}$ has codimension at most one.
\end{proof}

\begin{theorem}\label{rationalmate}
Let $B=(b_1,\dots,b_k)$ be rational functions such that $\log(1-\sum_i|b_i|^2)\in L^1(\T)$ and $B(0)=0$, but $B\ne 0$. Let $T=(M_z,\HH(B))$ and $\HN=[\ran \Delta]_{T^*}=[\ran \Delta]_{L}$.

If $a$ is the mate of $B$, then  $$a(z)=a(0)\ \frac{p(\frac{1}{z})}{q(\frac{1}{z})},$$ where $p$ is the characteristic polynomial of $T^*|\HN$ and $q$ is the characteristic polynomial of $L|\HN$.
\end{theorem}
\begin{proof} Assume that the degree of $B$ is $n$. Then by Theorem \ref{ranDelta and rational} the dimension of $\HN$ equals $n$.  Let $p(z)=\prod_{j=1}^n (z-\overline{\lambda_j})$ be the characteristic polynomial of $T^*|\HN$ and $q(z)=\prod_{j=1}^n(z-\alpha_j)$ be the characteristic polynomial of $L|\HN$. Here $\lambda_1,\dots, \lambda_n\in \overline{\D}$ and $\alpha_1,\dots, \alpha_n \in \D$. Then $\HN$ has the form as in (\ref{N as rational functions}).

Write $\HM=\HN^\perp$. Then since $\dim \HM^\perp <\infty$ the subspace $\HM\ominus z\HM$ is 1-dimensional (see e.g. \cite{ARS02}, Lemma 2.1) and $M_z|\HM$ is unitarily equivalent to the unilateral shift of multiplicity 1. Let $\varphi\in \HM\ominus z\HM$ be a unit vector, then $\{z^n\varphi\}$ forms an orthonormal basis for $\HM$ and hence $P_{\HM}K^B_w(z)=\frac{ \varphi(z)\overline{\varphi(w)}}{1-z\overline{w}}$. Since $K^B_w(z)=P_{\HN}K^B_w(z)+ P_{\HM}K^B_w(z)$ we conclude that
$$1-\sum_{i=1}^k|b_i(z)|^2= (1-|z|^2)\|P_\HN K^B_z\|^2 + |\varphi(z)|^2$$ for all $z\in \D$. Since all functions in $\HN$ are bounded and $\HN$ is finite dimensional we let $|z|\to 1$ and obtain that $|\varphi|^2=|a|^2$ on $\T$.

Let $C= (I-P_{\HN})M_zP_{\HN}$. From Theorem \ref{Wold-type} we know that $\rank C =\dim \ker M_z^*=1$ and we know $\ran C \subseteq \HM\ominus z\HM= \ker (M_z|\HM)^*$ since $M_z$ expands the norm (see e.g. the proof of Theorem \ref{Wold-type}). Hence there is $f_1\in \HN$ such that $Cf_1=\varphi$. Then there is $f_2\in \HN$ such that $\varphi=zf_1-f_2$. Considering the form of the functions in $\HN$ (see \ref{N as rational functions}), we conclude that $\varphi(z)= \frac{h(z)}{\prod_{j=1}^n(1-\alpha_j z)}$ for some polynomial $h$ of degree $\le n$.

Thus, $\varphi \in H^\infty$ and in fact $\varphi f\in \HM$ for all $f\in \HH(B)$. In particular, if $g\in \HN$, then $hg\in \HM$ or $P_{\HN}h(M_z)|\HN=0$. Let $\tilde{h}(z)=\overline{h(\overline{z})}$, then $\tilde{h}(M_z^*)|\HN=0$. This implies that the minimal polynomial of $M_z^*|\HN$ divides $\tilde{h}$. But by Lemma \ref{eigenspaces of M^*} the minimal polynomial equals the characteristic polynomial and it has degree $n$. Hence $\tilde{h}$ must be a multiple of $p$, and that implies that $\varphi(z) = \gamma \frac{\prod_{j=1}^n(z-\lambda_j)}{\prod_{j=1}^n(1-\alpha_jz)}$ for some $\gamma\in \C$. Now $a$ is outer, but has the same modulus as $\varphi$ on $\T$, hence $a(z)= a(0)\frac{\prod_{j=1}^n(1-\overline{\lambda_j}z)}{\prod_{j=1}^n(1-\alpha_jz)}$. This proves the theorem.
\end{proof}
Theorem \ref{characterizationmate} now follows from Theorems \ref{ranDelta and rational} and \ref{rationalmate}.

\begin{remark}\label{equalityDegrees} Note the relationship between the functions $a$ and $\varphi$ in the proof of Theorem \ref{rationalmate}: $\varphi = S a$ for some finite Blaschke product   $S$ and the complex conjugates of the zeros of $S$ must be in $\sigma(T^*|\HN)$. Thus,
if $\sigma(T^*|\HN) \subseteq \T$, then $\deg a=\deg B = \dim \HN$ (since there is no cancellation of linear factors of $\tilde{p}$ and $\tilde q$), and $\varphi$  must be a constant multiple of $a$.
\end{remark}

\section{Expansive $m$-isometries, general considerations}\label{generalExpansiveMIsos}
Recall from the Introduction that a bounded linear operator $T$ on a Hilbert space $\HH$ is an
$m$-isometry for some positive integer $m$ if
\[
\beta_{m}(T)=\sum\limits_{k=0}^{m}(-1)^{m-k}\binom{m}{k}y^{k}x^{k}|_{y=T^{\ast},x=T}%
=\sum\limits_{k=0}^{m}(-1)^{m-k}\binom{m}{k}T^{\ast k}T^{k}=0.
\]
The first main result of this section will be a theorem that will imply one of the directions of Theorem \ref{twonisometry} (b). We start with a lemma.

\begin{lemma}\label{beta-n} Let $n\in \N$, $w\in \T$, $T,P\in \HB(\HH)$, with $P\ge 0$. Then
$$\sum_{k=0}^{n-1}\binom{n-1}{k}(-1)^{n-1-k}{T^*}^kPT^k = \sum_{k=0}^{n-1} \binom{n-1}{k} (T^*-\overline{w})^k PT^k (\overline{w}T-I)^{n-1-k}.$$
\end{lemma}
\begin{proof} The intuition for this formula comes from the following application of Agler's hereditary functional calculus:
\begin{align*}
 (yx-1)^{n-1}(P)|_{y=T^*,x=T} &= [(y-\overline{w})x+(\overline{w} x-1)]^{n-1}(P)|_{y=T^*,x=T}\\
& = \sum_{k=0}^{n-1} \binom{n-1}{k} (T^*-\overline{w})^kP  T^k (\overline{w} T-1)^{n-1-k}.\end{align*}

But it can also be proved by induction. For $n\in \N$ set $$\gamma_n(P)= \sum_{k=0}^{n-1}\binom{n-1}{k}(-1)^{n-1-k}{T^*}^kPT^k.$$
One easily checks that in case $n=1$. Thus,  assume that the formula holds for some $n\ge 1$.
Then \begin{align*}
\gamma_{n+1}(P)&=\sum_{k=0}^{n}\binom{n}{k}(-1)^{n-k}{T^*}^kPT^k\\
&= T^*\gamma_n(P)T-\gamma_n(P)\\
&=(T^*-\overline{w})\gamma_n(P)T+\gamma_n(P)(\overline{w}T-I).
\end{align*}
At this point we leave the easy remaining details to the reader.
\end{proof}
 \begin{theorem}\label{DeltajImplies2miso} Let $m\in \N$ and let $T\in \HB(\HH)$ be such that $\Delta=T^*T-I\ge 0$.

 If there are $w_1, w_2,  \dots \in \T$ and positive operators $\Delta_1, \Delta_2, \dots$ such that $\Delta=\sum_{i\ge 1} \Delta_i$ and $(T^*-\overline{w}_j)^m \Delta_j =0$ for each $j\ge 1$, then $T$ is a $2m$-isometry.
 \end{theorem}

\begin{proof} Note that the sum converges in the strong operator topology. We use Lemma \ref{beta-n} with $w_i$ and $P=\Delta_i$ for each $i$, and we obtain
\begin{align*}\beta_{2m}(T) &=  \sum_{j=0}^{2m-1} \binom{2m-1}{j} (-1)^{2m-1-j} {T^*}^j\Delta  T^j \\
&= \sum_{i\ge 1} \sum_{j=0}^{2m-1} \binom{2m-1}{j}  (-1)^{2m-1-j} {T^*}^j\Delta_i  T^j \\
&= \sum_{i\ge 1} \sum_{j=0}^{2m-1} \binom{2m-1}{j} (T^*-\overline{w_i})^j\Delta_i  T^j (\overline{w}_i T-1)^{2m-1-j}\\
&=0,
\end{align*}
since either $j \geq m$ or $2m-1-j \geq m$. Hence $T$ is $2m$-isometry.
\end{proof}

Next we will see that for expansive $m$-isometries the finiteness of $\dim [\ran \Delta]_{T^*}$ follows from $\Delta$ having finite rank.

\begin{lemma} \label{rank and dim N} If $T$ is an $m$-isometry, if $\Delta=T^*T-I\ge 0$, and if $\Delta$ has finite rank, then $$\rank \Delta\le \dim [\ran \Delta]_{T^*} \le \begin{cases}\frac{m}{2} \rank \Delta,& m ~\text{is even}\\
\frac{m-1}{2}\rank \Delta,& m ~\text{is odd}\end{cases}.$$
\end{lemma}
\begin{proof} The case $m=1$ is trivial, we assume $m\ge 2$. Then
 $$0=\beta_m(T)=T^*\beta_{m-1}(T)T-\beta_{m-1}(T)=\sum_{j=0}^{m-1}(-1)^{m-1-j} \left(\begin{matrix}m-1\\j\end{matrix}\right){T^*}^j\Delta T^j.$$ Thus we have
 $$\sum_{j \text{ even}}^{m-1} \left(\begin{matrix}m-1\\j\end{matrix}\right){T^*}^j\Delta T^j = \sum_{j\text{ odd}}^{m-1} \left(\begin{matrix}m-1\\j\end{matrix}\right){T^*}^j\Delta T^j.$$ Hence  we have
\begin{align*}\bigvee_{j \text{ even}}^{m-1} \ran {T^*}^j\Delta T^j = \bigvee_{j \text{ odd}}^{m-1} \ran {T^*}^j\Delta T^j. \end{align*}
So if $\Delta = \sum_{k=1}^n f_k\otimes f_k$, then when $m=2l$,
$$\bigvee_{i=1}^n\{T^*f_i,T^{*3}f_i\cdots, T^{*(2l-1)}f_i\} \subseteq \bigvee_{i=1}^n\{f_i,T^{*2}f_i\cdots, T^{*(2l-2)}f_i\},$$
and when $m=2l+1$,
$$\bigvee_{i=1}^n\{f_i,T^{*2}f_i\cdots, T^{*(2l)}f_i\} \subseteq \bigvee_{i=1}^n\{T^*f_i,T^{*3}f_i\cdots, T^{*(2l-1)}f_i\}.$$
This implies that the sets on the right hand side of the above two inclusions are $T^*$-invariant and hence
$$\dim [\ran \Delta]_{T^*} \le
\begin{cases}\frac{m}{2} \rank \Delta,& m ~\text{is even}\\
\frac{m-1}{2}\rank \Delta,& m ~\text{is odd}\end{cases}.$$
\end{proof}

The following construction  will be crucial for the rest of this section.

\begin{lemma} Let  $T\in \HB(\HH)$ be  such that $\Delta=T^*T-I\ge 0$. Set $\HN=[\ran \Delta]_{T^*}$ and assume that $\dim \HN =n< \infty$.

 If $c=(c_0,\dots ,c_{n-1})$, where $c_j>0$ for each $j, 0\le j\le n-1$, then $$\|x\|_c^2=\sum_{j=0}^{n-1} c_j\la \Delta T^jx,T^jx\ra$$ defines a Hilbert space norm on $\HN$.
\end{lemma}
\begin{proof} Fix a $c$ as in the hypothesis of the lemma and write $B_c=\sum_{j=0}^{n-1} c_j{T^*}^j \Delta T^j$. Thus, $B_c\ge 0$ and $\|x\|^2_c=\la B_cx,x\ra$. In order to show that $\|\cdot\|_c$ is a norm on $\HN$ it suffices to show that $\HN \cap \ker B_c=(0)$. Let $\Delta=\sum_{i=1}^N f_i\otimes f_i$. Then $$\ran B_c = \bigvee_{i=1}^N \bigvee_{j=0}^{n-1} {T^*}^jf_i.$$ Note that since $\dim \HN=n$ we have that for each fixed $i$ the vector ${T^*}^nf_i$ must be a linear combination of $f_i, T^*f_i, \dots, {T^*}^{n-1} f_i$. This means that $\bigvee_{j=0}^{n-1}{T^*}^jf_i$ is invariant for $T^*$, and it follows that $\ran B_c$ is $T^*$-invariant. Since $\ran B_c $ contains $\bigvee_{i=1}^N \{f_i\}=\ran \Delta$ we conclude that $\HN \subseteq \ran B_c$. This implies $\HN \cap \ker B_c=(0)$.
\end{proof}
If $T$ is an operator as in the previous lemma, and if $c_0, \dots, c_{n-1}>0$, then we will write $\HN_c=(\HN, \|\cdot\|_c)$ and $A_c$ for the operator $A=P_{\HN}T|\HN$ as it acts on $\HN_c$.
\begin{theorem}\label{A and T} Let $N\in \N$, $N\ge 2$, and  $T\in \HB(\HH)$ be  such that $\Delta=T^*T-I\ge 0$. Set $\HN=[\ran \Delta]_{T^*}$ and assume that $\dim \HN =n< \infty$.

Then $T$ is an $N$-isometry, if and only if $A_c$ is an $(N-1)$-isometry for all $c=(c_0,\dots, c_{n-1})$, $c_j>0$.
\end{theorem}
\begin{proof} Write $D=\Delta^{1/2}$, and recall from the proof of Lemma \ref{rank and dim N} that $$\beta_N(T)= \sum_{k=0}^{N-1}(-1)^k \left(\begin{matrix} N-1\\k\end{matrix}\right){T^*}^k\Delta T^k.$$ Thus, for all $x\in \HH$ we have
$$\la \beta_N(T)x,x\ra= \sum_{k=0}^{N-1}(-1)^k \left(\begin{matrix} N-1\\k\end{matrix}\right)\|DT^kx\|^2.$$ We will use this identity repeatedly.

Fix $c=(c_0,\dots, c_{n-1})$ as in the lemma, then for $x\in \HN$
$$\|A^k x\|^2_c= \sum_{j=0}^{n-1}c_j\|DT^{j+k}x\|^2$$ since $A$ is the compression of $T$ to the semi-invariant subspace $\HN$ and $D=DP_{\HN}$. This implies that
\begin{align*}\sum_{k=0}^{N-1}(-1)^k \left(\begin{matrix} N-1\\k\end{matrix}\right)\|A^k x\|^2_c&= \sum_{j=0}^{n-1}c_j\sum_{k=0}^{N-1}(-1)^k \left(\begin{matrix} N-1\\k\end{matrix}\right)\|DT^{j+k}x\|^2\\
&=\sum_{j=0}^{n-1}c_j \la \beta_N(T)T^jx,T^jx\ra\end{align*}
Hence if $T$ is an $N$-isometry, then $A_c$ is an $N-1$-isometry on $\HN_c$ for all tuples $c$ of positive reals.

Conversely, if $A_c$ is an $N-1$-isometry on $\HN_c$ for all positive tuples $c$, then $$\sum_{j=0}^{n-1}c_j \la \beta_N(T)T^jx,T^jx\ra=0$$ for all $x\in \HN$ and all such $c$'s. That can only hold, if $\la \beta_N(T)x,x\ra=0$ for all $x\in \HN$. Thus
$$\sum_{k=0}^{N-1}(-1)^k \left(\begin{matrix} N-1\\k\end{matrix}\right)\|D T^kx\|^2=0$$ for all $x\in \HN$.

 If $x\in \HH$, then $x=x_1+x_2$ with $x_1\in \HN$ and $x_2\in \HN^\perp$. Since $\HN^\perp$ is $T$-invariant we have $DT^k x_2=0$ for each $k$ and we conclude that $\beta_N(T)=0$, i.e. $T$ is an $N$-isometry.
\end{proof}

 We now show  Theorem \ref{twonisometry} (a).

\begin{corollary}\label{OddIsEven} Let $m\in \N$, and $T\in \HB(\HH)$ such that $\Delta=T^*T-I\ge 0$ and has finite rank.

If $T$ is a $2m+1$-isometry, then it is a $2m$-isometry and $\sigma(T^*|\HN)\subseteq \T$, where $\HN=[\ran \Delta]_{T^*}$.
\end{corollary}
\begin{proof} We assume that $T$ is  a norm expansive $2m+1$-isometry such that $\Delta$ has finite rank, and as before we write $A=P_{\HN} T|\HN$. Then by Lemma \ref{rank and dim N} we have $\dim \HN <\infty$. Thus, by Theorem \ref{A and T} we conclude that $A_c$ is a $2m$-isometry for each  tuple $c$ of positive reals. Agler and Stankus showed that the spectrum of any $n$-isometry is either the closed disc or a subset of the unit circle (\cite{AS956}, Lemma 1.21). Since $\HN_c$ is finite dimensional, this implies that $\sigma(A)\subseteq \T$. Hence $A$ is invertible, and then Proposition 1.23 of \cite{AS956} implies that $A$ must be a $2m-1$-isometry in the norm $\|\cdot\|_c$. This is true for all tuples $c$ of positive reals, hence another application of Theorem \ref{A and T} shows that $T$ is a $2m$-isometry.

We saw that $\sigma(A)\subseteq \T$, hence $\sigma(T^*|\HN)=\sigma(A^*)\subseteq \T$.
\end{proof}

For $A\in \HB(\HN)$ and $w\in \C$ let $$\HN_w=\HN_w(A)= \bigvee_{n\ge 0} \ker (A-w)^n,$$ so that for $w\in \sigma_p(A)$ the space $\HN_w$ is the root subspace of $A$ corresponding to the eigenvector $w$, and if $\HN$ is finite dimensional, then it is clear that $\HN=\bigvee_{w\in \sigma_p(A)} \HN_w$.

We will need the following result, which is a special case of \cite{AHS98}, Lemma 19, also see \cite{JJS20}, Proposition 6.3, or \cite{BMN13}, Theorem 2.7.

\begin{theorem}\label{finite dim m isos}   Let $\HN$ be a finite dimensional Hilbert space, let  $A\in \HB(\HN)$ with $\sigma(A)=\{w_1, \dots, w_n\}$, then $A$ is a  $2m-1$-isometry, if and only if \begin{enumerate}
\item $\sigma(A) \subseteq \T$,
\item $\HN =\bigoplus_{k=1}^n \HN_{w_k}$,
\item $m\ge \min\{i: \ker (A-w)^i = \HN_w\}$ for each $w\in \sigma(A)$.
\end{enumerate}
If  $A$ is a  $2m-1$-isometry, then it is a strict $2m-1$-isometry, if and only if $m=\max_{w\in \sigma(A)}\min\{i: \ker (A-w)^i = \HN_w\}$.
\end{theorem}

The first main step of the converse of Theorem \ref{DeltajImplies2miso} if $\Delta$ has finite rank now follows easily.
\begin{theorem}\label{characterization1}
Let $m\in \N$, and let $T\in \HB(\HH)$ such that $\Delta = T^*T - I$ is positive and has finite rank. Let $\HN=[\ran \Delta]_{T^*}$ and $A=P_{\HN}T|\HN$.

If  $T$ is a strict $2m$-isometry, then
 $\Delta\HN_w(A) \perp \HN_z(A)$ for all $w\ne z$, and
 $m=\max_{w\in \sigma(A)}\min\{i: \ker (A-w)^i = \HN_w\}$.

Furthermore,  $\Delta^{1/2}x\ne 0$ for each nonzero eigenvector $x$ of $A$.
\end{theorem}
{\bf Remarks:} 1. The last sentence of the theorem combined with $\Delta \HN_w(A)\perp \HN_z(A)$ for all $z\ne w$ implies that if $T$ is a $2m$-isometry, then  $\card \sigma(A)\le \rank \Delta$.

2. If $\dim \HN <\infty$ and $\sigma(A)= \{w_1, \dots, w_k\}$, then let $p(z)= \prod_{j=1}^k(z-w_j)^{m_j}$ be the minimal polynomial of $A$. Then the condition $m=\max_{w\in \sigma(A)}\min\{i: \ker (A-w)^i = \HN_w\}$ is easily seen to be equivalent to $m=\max\{m_j: j=1, \dots, k\}$.

In particular, if $\dim \ker T^*=1$, then the condition  turns out to be equivalent to $m=\max_{w\in \sigma(A)} \dim \HN_w$. Indeed, since $\dim \HN<\infty,$
 Theorem \ref{representationVR} implies that $T = V \oplus R$, where $V$ is isometric and $R$ is unitarily equivalent to $(M_z,\HH(B))$ for some rational $B \in \HS(\C^k, \C)$. Then it follows from Lemma \ref{eigenspaces of M^*} that the minimal and characteristic polynomials of $A^*$ agree. Then the same is true for $A$  and
  $\dim \HN_w(A)=\min\{i: \ker (A-w)^i = \HN_w\}$ for all $w\in \sigma(A)$.

\begin{proof} Since $T$ is a strict $2m$-isometry Lemma \ref{rank and dim N} implies that the space $\HN$ is finite dimensional. Let $\dim \HN=N$. As in Theorem \ref{A and T} we consider tuples $c=(c_0,c_1,\dots,c_{N-1})$ with  $c_j>0$ for $j=0,\dots , N-1$, and we consider the Hilbert space $\HN_c$ which equals $\HN$, but with the norm $\|\cdot \|_c$. We write $A_c$ to denote the Hilbert space operator $A$ acting in $\HN_c$. According to Theorem \ref{A and T} $A_c$ is a $2m-1$-isometry for each such $c$. Furthermore, if $A_c$ was a $2m-2$-isometry for every $c$, then by  Theorem \ref{A and T}  $T$ would be a $2m-1$ isometry. Thus,  since $T$ is a strict $2m$-isometry, there is a tuple $\tilde{c}$ such that the operator $A_{\tilde{c}}$ is a strict $2m-1$-isometry.

Write $\sigma(A)=\{w_1, \dots, w_n\}$. Then $\sigma(A_c)=\sigma(A)$ and $A_c$ and $A$ are similar for all $c$ and all $1 \le i \le n$, hence the condition that $m=\max_{w\in \sigma(A)}\min\{i: \ker (A-w)^i = \HN_w\}$ follows immediately from  Theorem \ref{finite dim m isos}.  Condition (ii) of Theorem \ref{finite dim m isos} implies that for each $c$ we have
  $\HN_{w_k} \perp \HN_{w_j}$ with respect to $\la \cdot, \cdot \ra_c$ whenever $j\ne k$. Thus, if $j \ne k$ and if $x\in \HN_{w_j}$, $y\in \HN_{w_k}$, then
 $$\sum_{i=0}^{N-1} c_i \la \Delta T^ix,T^iy\ra =0.$$
 The spaces $\HN_{w}$ do not depend on the $c_i$'s, hence we conclude that each term in the sum had to be 0. In particular, $\la \Delta x,y\ra =0$. This shows $\Delta\HN_w(A) \perp \HN_z(A)$ for all $w\ne z$.

If $x\in \HN$ is an eigenvector for $A$, say $Ax=\lambda x$ for some $\lambda\in \C$, then $|\lambda|^{2i}\|\Delta^{1/2}x\|^2= \|\Delta^{1/2}A^ix\|^2=\|\Delta^{1/2}T^ix\|^2$. Hence for any $c$ as above we have $\|x\|_c^2= \sum_{i=0}^{2m}{c_i}|\lambda|^{2i}\|\Delta^{1/2}x\|^2$. Since $\| \cdot\|_c$ is a norm on $\HN$ we conclude that if $x\ne 0$, then $\Delta^{1/2}x\ne 0$.
\end{proof}

With the help of a lemma from linear algebra we will transform the condition  of the previous theorem into an equivalent condition that will be easy to apply later.
Let $T\in \HB(\HH)$, then if $f$ is analytic in a neighborhood of $\sigma(T)$ the Riesz Dunford functional calculus is defined by $$f(T)=\frac{1}{2\pi i} \int_\gamma f(z)(z-T)^{-1} dz,$$ where $\gamma$ is a curve that surrounds $\sigma(T)$ once in the positive direction. If $\tilde{f}(z)=\overline{f(\overline{z})}$, then it is well-known that $f(T)^*=\tilde{f}(T^*)$.

If $A$ is an $n\times n$ matrix with $\sigma(A)=\{w_1, \dots, w_k\}$, then let $f_i$ be 1 in a neighborhood of $w_i$ and 0 in an open set that includes all $w_j$ with $j \ne i$. Then $f_i(A)$ satisfies $f_i(A)^2=f_i(A)$, it has range equal to $\HN_{w_i}(A)$, and $\sum_{i=1}^k f_i(A)=I$. Also note that $\tilde{f}_i$ is 1 in a neighborhood of $\overline{w}_i$ and 0 elsewhere, so $f_i(A)^*=\tilde{f}(A^*)$ is an idempotent with  range $\HN_{\overline{w}_i}(A^*)$.

\begin{lemma}\label{existencedelta}
Let $A$ be an $n\times n$ matrix with $\sigma(A)=\{w_1, \dots, w_k\}$ and let $\Delta$ be a positive definite $n\times n$ matrix.

Then $\Delta \HN_{w_i}(A) \perp \HN_{w_j}(A)$ for all $i \ne j$, if and only if $$\Delta= \sum_{j=1}^k \Delta_j,$$
 where $\Delta_j\ge 0 $ and $\ran \Delta_j  \subseteq \HN_{\overline{w}_j}(A^*)$ for each $j=1, \dots, k$.
\end{lemma}
Note that if the conditions of the lemma are satisfied, then the geometry of the spaces $\HN_{\overline{w}_j}(A^*)$ implies that $\dim \HN= \sum_{j=1}^k \dim \HN_{\overline{w}_j}(A^*)$ and $\rank \Delta = \sum_{j=1}^k \rank \Delta_j$.
\begin{proof} Note that $f_i(A)f_j(A)=0$ and $\tilde{f}_i(A^*)\tilde{f}_j(A^*)=0$ for all $i \ne j$.

First suppose that $\Delta \HN_{w_i}(A) \perp \HN_{w_j}(A)$ for all $i \ne j$. Then for all $i\ne j$ we have $\tilde{f}_j(A^*)\Delta f_i(A)=0$.
Hence $$\Delta =\left(\sum_{j=1}^k \tilde{f}_j(A^*)\right) \Delta \left(\sum_{i=1}^k f_i(A)\right)=\sum_{j=1}^k \tilde{f}_j(A^*)\Delta f_j(A).$$ Set $\Delta_j=\tilde{f}_j(A^*)\Delta f_j(A)= f_j(A)^* \Delta f_j(A)$. Then $\Delta_j \ge 0$, $\Delta=\sum_{j=1}^k\Delta_j$ and $\ran \Delta_j \subseteq \ran \tilde{f}_j(A^*)=\HN_{\overline{w}_j}(A^*)$.

Conversely, assume that $\Delta= \sum_{j=1}^k \Delta_j$ for  nonnegative matrices $\Delta_j$ that satisfy  $\ran \Delta_j  \subseteq \HN_{\overline{w}_j}(A^*)$ for each $j=1, \dots, k$. Let $i\ne j$ and let $x\in \HN_{w_i}(A), y\in \HN_{w_j}(A)$. Then $x=f_i(A)x$ and $y=f_j(A)y$, and $\Delta_m=\tilde{f}_m(A^*)\Delta_m$ for each $m$, hence
\begin{align*}\la \Delta x,y\ra &=\sum_{m=1}^k \la  \tilde{f}_m(A^*)\Delta_m x, f_j(A)y\ra\\
&=\sum_{m=1}^k \la \Delta_m x, f_m(A)f_j(A)y\ra\\
&=\la \Delta_j x, y\ra\\
&=\la x, \Delta_j y\ra\\
&=\la f_i(A) x, \tilde{f}_j(A^*) \Delta_j y \ra\\
&=\la f_j(A) f_i(A)x,y\ra\\
&=0\end{align*}
Hence $\Delta \HN_{w_i}(A) \perp \HN_{w_j}(A)$ for all $i\ne j$.
\end{proof}

Finally we can prove the remaining direction of Theorem \ref{twonisometry} (b).

\begin{theorem}\label{characterizationR}
Let $T\in \HB(\HH)$ be a $2m$-isometry such that $\Delta=T^*T-I$ is positive and has finite rank.

Then there are $w_1,\dots, w_k\in \T$ and positive operators $\Delta_1, \dots ,\Delta_k$ such that
 $\Delta= \sum_{j=1}^k \Delta_j$ and  $(T^*-\overline{w}_j)^m\Delta_j =0$ for each $j=1,\dots , k$.

If  $p(z)=\prod_{j=1}^k(z-\overline{w}_j)^{m_j}$ is the minimal polynomial of $T^*|\HN$, then $T$ is a strict $2m_0$-isometry, where $m_0=\max\{ m_j: j=1, \dots, k\}$.
\end{theorem}
\begin{proof}
This theorem follows from Theorem \ref{characterization1}, Lemma \ref{existencedelta}, and the observation that $T^*|\HN=A^*$, where $A=P_\HN T|\HN$.
\end{proof}

\begin{corollary}\label{defectrankone}
Let $m\in \N$,  $T\in \HB(\HH)$ be such that $\Delta = T^*T - I = \eta \otimes \eta$. Let $\HN = [\eta]_{T^*}$ Then the following are equivalent:

(a) $T$ is a strict $2m$-isometry,

(b) $\dim \HN = m$ and there exists $w \in \T$ such that $(T^*-\overline{w})^m \eta = 0$,

(c) there exists $w \in \T$ such that $p(z)=(z-\overline{w})^m$ is the characteristic polynomial of $T^*|\HN$.
\end{corollary}

\begin{proof}
Since $\eta$ is a cyclic vector for $T^*|\HN$ it is clear that (b) and (c) are equivalent, and that the characteristic polynomial of $T^*|\HN$ equals its minimal polynomial. Thus, if (b) holds then Theorems 1.4 or 8.10  ensure that $T$ is a strict $2m$-isometry, i.e. condition (b) implies condition (a).

In order to prove (a) $\Rightarrow$ (b)  we suppose that $T$ is a strict $2m$-isometry. Then Theorem 1.4 implies that there exists $w \in \T$ such that $(T^*-\overline{w})^m \eta = 0$. This implies that $\dim \HN \le m$. If $\dim \HN < m$, then the minimal polynomial of $T^*|\HN$ is not $(z-\overline{w})^m$, which in turn implies that $T$ is not a strict $2m$-isometry. Hence $\dim \HN = m$.
\end{proof}

\section{Expansive $2m$-isometries, $\rank \Delta = 1$} \label{SectionLocalDiri}
Let $m \in \N$. If $w\in \T$, then we define the local Dirichlet space of order $m$ at $w$ by
\begin{align*}\HD_w^m=\{p+(z-w)^mg: p \text{ is a polynomial of degree }<m \text{ and } g\in H^2\}.\end{align*}
One easily checks that if  $f\in \HD_w^m$, then the polynomial $p$ of degree $<m$  and the function $g\in H^2$ such that $f=p+(z-w)^m g$ are unique. Thus, if $f\in H^2$, then $f\in \HD_w^m$ if and only if there is a (unique) polynomial $p$ of degree $<m$ such that $g=(f-p)/(z-w)^m \in H^2$. We define the local Dirichlet integral of order $m$ of $f$ at $w$ by
$$D_w^m(f)=\inf\{\|\frac{f-p}{(z-w)^m}\|^2_{H^2}: p \text{ is a polynomial of degree }<m\}.$$

It is clear that if $f\in \HD_w^m$ extends to be analytic in a neighborhood of $w$, then
\begin{align} \label{localDiri}D_w^m(f)= \int_{|z|=1} \left| \frac{f(z)-T_{m-1}(f,w)(z)}{(z-w)^m}\right|^2\frac{|dz|}{2\pi},\end{align}
where $T_{m-1}(f,w)$ be the $(m-1)$-th order Taylor polynomial of $f$ at $w$, $$T_{m-1}(f,w)(z)=\sum_{j=0}^{m-1} \frac{f^{(j)}(w)}{j!}(z-w)^j.$$
\begin{lemma} \label{nontangential} Let $m\in \N$, $w\in \T$, and $f\in H^2$, then $f\in \HD_w^m$, if and only if  for each $j=0,\dots, m-1$ the function $f^{(j)}$ has nontangential limit $b_j$ at $w$, and $D_w^m(f)$ as defined in (\ref{localDiri}) with $T_{m-1}(f,w)(z)=\sum_{j=0}^{m-1} \frac{b_j}{j!}(z-w)^j$ is finite.\end{lemma}
\begin{proof} Note that $(z-w)^m$ is an outer function. Thus, if $f\in H^2$ and if $p$ is any polynomial, then $(f-p)/(z-w)^m$ is in the Smirnov class, and hence it will be in $H^2$, whenever it is in $L^2(\T)$.

If $f \in \HD_w^m$, then $f = p+(z-w)^m g$ for some $g \in H^2$ and a  polynomial $p$ of degree $< m$. Let $i$ be an integer with $0\le i\le m-1$. Then since $g \in H^2$, we have
$$g^{(i)}(z) = \langle g, \frac{{\partial}^i k_z^{H^2}}{\partial \overline{z}^i}\rangle_{H^2}.$$
Thus,
$$|g^{(i)}(z)| \precsim \frac{1}{(1-|z|)^{i+1/2}}.$$
Then in the region
$$\Gamma_\alpha(w) = \{z \in \D: |z-w| < \alpha (1-|z|)\}$$
the function $((z-w)^mg)^{(j)}(z), 0 \leq j \leq m -1$, goes to zero as $z \rightarrow w$.
Thus, for each $j=0,\dots, m-1$ the function $f^{(j)}$ has nontangential limit equal to $b_j=p^{(j)}(w)$ at $w$. Hence $p(z)=\sum_{j=0}^{m-1}\frac{b_j}{j!}(z-w)^j=T_{m-1}(f,w)(z)$. \end{proof} If $f\in \HD_w^m$, then for $0\le j\le m-1$ we will write $f^{(j)}(w)$ to denote the  limit of $f^{(j)}(z)$ as $z\to w$ nontangentially.
We define a norm on $\HD_w^m$ by
$$\|f\|^2=\|f\|^2_{H^2}+D_w^m(f).$$
We note that if $m=1$, then we obtain the local Dirichlet integral $D_w^1(f)=D_w(f)$ that was introduced in \cite{RS91}.
\begin{lemma} \label{boundedFunctional} Let $m\in \N$, $w\in \T$. Then for each $0\le j\le m-1$ the functional $f \to f^{(j)}(w)$ is bounded on $\HD_w^m$. \end{lemma}
\begin{proof} Note that by Lemma \ref{nontangential} we have $\sup \{|f^{(j)}(z)|: z\in \Gamma_\alpha(w)\}<\infty$ for every $f \in \HD_w^m$. Thus, the lemma  follows from the uniform boundedness principle and Lemma \ref{nontangential}.\end{proof}

\begin{lemma}\label{DeltaLocalDiri} Let $f\in H^2$. Then $f\in \HD_w^m$ if and only if $zf \in \HD_w^m$. Furthermore, if $g\in H^2$ such that for $z\in \D$ we have  $$f(z)-(z-w)^mg(z)= \sum_{j=0}^{m-1}a_j(z-w)^j,$$ then
$$D_w^m(zf)=D_w^m(f)+|a_{m-1}|^2.$$\end{lemma}
\begin{proof}  If $f= \sum_{j=0}^{m-1}a_j(z-w)^j + (z-w)^mg$, then $$zf= a_0w+ \sum_{j=1}^{m-1}(a_{j-1}+a_jw)(z-w)^j + (z-w)^m(a_{m-1}+zg).$$ This implies $D_w^m(zf)= \|a_{m-1}+zg\|^2_{H^2}= |a_{m-1}|^2+D_w^m(f).$ Since by Lemma \ref{boundedFunctional} the functional $f \to a_{m-1}$ is bounded on $\HD_w^m$, it shows that $M_z$ is bounded on $\HD_w^m$ and that it expands the norm. Similarly, one sees that $zf \in \HD_w^m$ implies $f\in \HD_w^m$.
\end{proof}

It follows that $T=(M_z,\HD^m_w)$ is bounded, norm-expansive, analytic, and satisfies $\dim \ker T^*=\rank (T^*T-I)=1$, see Lemma 2.1 of \cite{Ri87}. Thus, by Theorem \ref{expandingOp} $T$ is unitarily equivalent to $(M_z,\HH(b))$ for some non-extremal $b$ in the unit ball of $H^\infty$. If we assume that $b(0)=0$, then the constant functions form $\ker M_z^*$ in both cases, and hence we must have $\HH(b)=\HD_w^m$ with equality of norms. It will follow from the next theorem that $b(z)=z^m/q(z)$, where $q$ is a polynomial of degree $m$ that has no zeros in $\D$ and such that $|q(z)|^2=1+|z-w|^{2m}$ for all $z\in \T$. The following is a version of Theorem \ref{2n-iso}.
\begin{theorem}\label{Thm1.1again} Let $b$ be non-extremal in the unit ball with $b(0)=0$, and let $m\in \N$. Then the following are equivalent:
\begin{enumerate}[label=\alph*)]
\item $(M_z,\HH(b))$ is a strict $2m$-isometry,
\item $b$ is rational of degree $\le m$ and there is $w\in \T$ and a polynomial $q$ of degree $\le m$ with no zeros in $\overline{\D}$ and such that the mate $a$ of $b$ is of the form $a(z)=\frac{(z-w)^m}{q(z)}$.
 \item there is $w\in \T$ and a polynomial $p$ of degree $< m$ such that $p(w) \ne 0$ and
 $$\|f\|^2_{\HH(b)}=\|f\|^2_{H^2}+D_w^m(pf).$$
\end{enumerate}
In fact, if $(M_z,\HH(b))$ is a strict $2m$-isometry, then there is $w\in \T$ and there are polynomials $p$ and $q$ of degree $\le m$ such that
\begin{enumerate}
     \item $|q(z)|^2=|p(z)|^2+|z-w|^{2m}$ for all $z\in \T,$
     \item $p(w)\ne 0$, $p(0)=0$
     \item $q(z) \ne 0$ for all $z\in \overline{\D}$,
     \item $b=p/q$,
     \item $\|f\|^2_{\HH(b)}=\|f\|^2_{H^2}+D_w^m(\tilde{p}f)$, where $\tilde{p}(z)=z^m\overline{p(\frac{1}{\overline{z}})}$.
    \end{enumerate}
    Note that $p(0)=0$ implies that degree $\tilde{p} \le m-1$.
\end{theorem}
\begin{proof} The conditions {\it (i)-(v)} will be verified while we are showing the equivalence of the conditions (a), (b) and (c). We start by proving the equivalence of (a) and (b). Set $T=(M_z,\HH(b))$, $\Delta=T^*T-I$, and $\HN=[\ran \Delta]_{T^*}$. Then $T$ is analytic, norm expansive, and satisfies $\dim \ker T^*=\rank \Delta=1$. Thus, by Corollary \ref{defectrankone}  $T$ is a strict $2m$-isometry, if and only if the characteristic polynomial of $T^*|\HN$ equals $(z-\overline{w})^m$ for some $w\in \T$.

Thus, the equivalence of (a) and (b) now follows from Theorem \ref{characterizationmate}.

Next we prove the equivalence of (b) and (c). We start by assuming condition (b) is satisfied. Then we may assume that $b=p/q$ for some polynomials $p$ and $q$ of degree $\le m$, and that the mate $a$ of $b$ is of the form $a(z)=(z-w)^m/q(z)$ for some $w\in \T$. The polynomial  $q$ has no zeros in the closed unit disc, and $|p(z)|^2+|z-w|^{2m}=|q(z)|^2$ for all $z\in \T$. This implies $p(w)\ne 0$.

 For $u\in H^\infty$ we will write $T_u$ for the multiplication operator on $H^2$, $T_uf=uf$ for all $f\in H^2$.

Now let $f\in \HH(b)$. The norm of $f$ in $\HH(b)$ is given by $$\|f\|^2_{\HH(b)}=\|f\|^2_{H^2}+ \|f^+\|^2_{H^2}$$ where $f^+$ is the unique function that satisfies $T_b^*f=T^*_af^+$, see \cite{Sa94}. Since $q$ has no zeros in the closed disc, this identity is equivalent to $T_p^*f=T_{(z-w)^m}^*f^+$. Thus, if $S$ denotes the unilateral shift, then $p(S)^*f-{(S-w)^m}^*f^+=0$.

Note that
$$p(S)^*f= \sum_{k=0}^m\overline{\hat{p}(k)}{S^*}^kf={S^*}^m\sum_{k=0}^m \overline{\hat{p}(k)}z^{m-k}f={S^*}^m(\tilde{p}f),$$ where as before $\tilde{p}(z)=z^m\overline{p(\frac{1}{\overline{z}})}$. Similarly,
${(S-w)^m}^*f^+={S^*}^m((1-\overline{w}z)^mf^+)$ and hence $\tilde{p}f-(1-\overline{w}z)^mf^+\in \ker {S^*}^m$. This implies that
there is a polynomial $R$ of degree $< m$ such that $$f^+= \frac{\tilde{p}f-R}{(z-w)^m}\in H^2.$$ Thus $\|f\|^2_{\HH(b)}=\|f\|^2_{H^2}+ D_w^m(\tilde{p}f).$
The hypothesis that $b(0)=0$ implies that $p(0)=0$ and that implies that the degree of $\tilde{p}$ is $< m$. Note that since $|w|=1$ we have $p(w)=0$, if and only if $\tilde{p}(w)=0$, and we had already noted that $p(w)\ne 0$. This proves $(b) \Rightarrow (c)$.

Finally suppose that (c) holds. Then for some $w\in \T$ and some polynomial $p$ of degree $<m$ with $p(w)\ne 0$ we have
$D^m_w(pf)=\|f^+\|^2_{H^2}$ for every $f \in \HH(b)$. Here $f^+$ is the unique $H^2$-function with $T^*_bf=T^*_af^+$.
 The definition of the local Dirichlet integral of order $m$ at $w$ implies that there must be a polynomial $R$ of degree $<m$ such that $(z-w)^mf^+= pf-R$. This implies that ${S^*}^m(pf-(z-w)^mf^+)=0$. Calculating as above, we see that this is equivalent to $$T_{\tilde{p}}^*f=\tilde{p}(S)^*f={(1-\overline{w}S)^m}^*f^+= (-1)^mw^mT_{(z-w)^m}^*f^+.$$ In this case $\tilde{p}(z)=z^m \overline{p(1/\overline{z})}$, so that the fact that degree $p <m$ implies that $\tilde{p}(0)=0$. Now by the Fej\'{e}r-Riesz Theorem there is a polynomial $q$ of degree $\le m$ such that $|z-w|^{2m}+|\tilde{p}(z)|^2=|q(z)|^2$ for all $z\in \T$. Since $p(w)\ne 0$ we have $\tilde{p}(w)\ne 0$ and hence we may assume that $q$ has no zeros in the closed unit disc. Then the adjoint Toeplitz operator $T_q^*$ is invertible. Now set $\tilde{b}(z)=\tilde{p}(z)/q(z)$, then by the properties of $q$ there must be a $c\in \T$ such that  $\tilde{a}(z)=c(z-w)^m/q(z)$ is the mate of $\tilde{b}$. The invertibility of $T_q^*$ implies that $T^*_{\tilde{b}}f= (-1)^mw^m c T^*_{\tilde{a}}f^+$. This implies that $\|f\|^2_{\HH(\tilde{b})}=\|f\|^2_{H^2}+\|f^+\|^2_{H^2}=\|f\|^2_{\HH(b)}$ for all $f \in \HH(b)$. Thus $b(z)= c' \tilde{b}(z)$ for some constant $c'\in \T$, see Lemma \ref{unitaryEquiv}. Hence $b$ is a rational function of degree $\le m$ and its mate is of the form $(z-w)^m/q(z)$ for some polynomial $q$ that has no zeros in the closed unit disc. Hence $(c) \Rightarrow (b)$.
\end{proof}

\begin{corollary} \label{equivalentNorms} Let $m\in \N$, $w\in \T$ and $p$ be a polynomial of degree $\le m-1$ with $p(w)\ne 0$, and let $b$ be in the unit ball of $H^\infty$ such that $\|f\|^2_{\HH(b)}=\|f\|^2_{H^2}+D^m_w(pf)$ for all $f\in \HH(b)$.

 Then $\HH(b)= \HD_w^m $ with equivalence of norms, i.e. there are constants $c,C>0$ such that
$$c (\|f\|^2_{H^2}+D_w^m(f)) \le \|f\|^2_{H^2}+D_w^m(pf) \le C(\|f\|^2_{H^2}+D_w^m(f)).$$
\end{corollary}
\begin{proof} We already showed in  Lemma \ref{DeltaLocalDiri} that $M_z$ acts boundedly on $\HD_w^m$, hence it is clear that $D_w^m(pf) \lesssim D_w^m(f)$ and hence $\HD_w^m \subseteq \HH(b)$.  Thus, in order to complete the proof it will suffice to show that $\HH(b) \subseteq D^m_w= {\mathcal{P}}_{m-1}\dotplus(z-w)^mH^2$. The norm inequality will then follow by a routine application of the Closed Graph Theorem.

Let $T=(M_z,\HH(b))$, $\Delta=T^*T-I$,  $\HN=[\ran \Delta]_{T^*}$, and $\HM=\HH(b)\ominus \HN$. In the proof of Theorem \ref{rationalmate} we showed that $\HH(b)= \varphi H^2 \oplus \HN$ for $\varphi \in \HM\ominus z\HM$, $\|\varphi\|=1$, and by Remark \ref{equalityDegrees} we have $\varphi = e^{it} a$. Thus, by Theorem \ref{Thm1.1again} and equation \ref{N as rational functions} $$\HH(b) = \frac{(z-w)^m}{q}H^2 \oplus \frac{1}{q}\ {\mathcal{P}}_{m-1},$$ where $q$ is a polynomial with no zeros in $\overline{\D}$. The result follows since multiplication by $q$ is an invertible operator on $\HH(b)$, see \cite{Sa94}, Section IV-5.
\end{proof}

\section{Finite rank expansive $2m$-isometries}

In this Section we will establish Theorem \ref{MainTheorem}.

We  start with a formula for the norm in $\HH(B)$, which in the current form is due to \cite{AlemanMalman}. Closely related formulas were also used in \cite{AlemanRichter96_97}, \cite{ARS96}, and \cite{AlemanFeldmanRoss}. In fact, with hindsight this formula can be used to motivate the definition of the first order local  Dirichlet integral.

\begin{lemma}\label{AlemanMalmanFormula} Let $B=(b_1,\dots, b_n)\in \HS(\C^n,\C)$ with $B(0)=0$ and such that $T=(M_z,\HH(B))$ is bounded. Set $\Delta=T^*T-I$ and  $D=\Delta^{1/2}$. Then for all polynomials $g\in \HH(B)$ we have
$$\|g\|^2_{\HH(B)}=\|g\|^2_{H^2}+ \int_{|\lambda|=1}\|D \frac{g-g(\lambda)}{z-\lambda}\|^2_{\HH(B)} \frac{|d\lambda|}{2\pi}.$$
\end{lemma}
\begin{proof} It is clear that $\HH(B)$ satisfies the conditions (A1')-(A3') of Aleman and Malman's paper \cite{AlemanMalman}. Hence their Proposition 2.5 applies, i.e.
$$\|g\|^2_{\HH(B)}=\|g\|^2_{H^2}+ \lim_{r\to 1}\int_{|\lambda|=1} [\|z \frac{g-g(r\lambda)}{z-r\lambda}\|^2_{\HH(B)} -r^2\| \frac{g-g(r\lambda)}{z-r\lambda}\|^2_{\HH(B)}]\frac{|d\lambda|}{2\pi}.$$
Thus, for polynomials $g$ this lemma follows from  the dominated convergence theorem.
\end{proof}
\begin{lemma}\label{AlemanMalmanFormulaLocalDiri} Let $m\in \N$, $w\in \T$, and $p$ be a polynomial of degree $\le m-1$. Then
$$D_w^m(pf)= \int_{|\lambda|=1} [D_w^m(zp \frac{f-f(\lambda)}{z-\lambda})-D_w^m(p \frac{f-f(\lambda)}{z-\lambda})] \frac{|d\lambda|}{2\pi}$$ for every polynomial $f$.
\end{lemma}
\begin{proof} First assume $p(w)\ne 0$ and define a norm on the polynomials by $\|f\|^2=\|f\|^2_{H^2}+D_w^m(pf)$. Then by Theorem \ref{Thm1.1again} $M_z$ extends to define a bounded $2m$-isometry $T=(M_z,\HH)$ on some space $\HH=\HH(b)$. Thus, in this case the lemma follows from Lemma \ref{AlemanMalmanFormula}.

If $p$ is identically 0, then the lemma is trivial. If $p\ne 0$, but $p(w)=0$, then $p(z)=(z-w)^kp_1(z)$ for some polynomial $p_1$ with $p_1(w)\ne 0$. Now we can apply the formula $D_w^m((z-w)^kp_1f)= D_w^{m-k}(p_1f)$ and apply the first case.
\end{proof}

\begin{lemma}\label{rank1andLocalDirichlet} Let $m\in \N$, $w\in \T$, $B\in \HS(\C^n,\C)$ be such that $B(0)=0$ and  $T=(M_z,\HH(B))$ is a bounded operator.

If  $f_0\in \HH(B)$ such that $(T^*-\overline{w})^mf_0=0$, but $(T^*-\overline{w})^{m-1}f_0\ne 0$, then there is a polynomial $p$ of degree $\le m-1$, with $p(w)\ne 0$, and such that $$D_w^m(zpg)-D_w^m(pg)=|\la g, f_0\ra|_{\HH(B)}^2$$ for every polynomial $g$.

Furthermore, with that $p$ we have $$D_w^m(pf)= \int_{|\lambda|=1} \left|\la \frac{f-f(\lambda)}{z-\lambda},f_0\ra_{\HH(B)}\right|^2 \frac{|d\lambda|}{2\pi}$$ for all polynomials $f$.
\end{lemma}
Note that if drop the part of the hypothesis that assumes $(T^*-\overline{w})^{m-1}f_0\ne 0$, then one can still deduce the existence of a polynomial $p$ which satisfies all parts of the conclusion except that possibly $p(w)=0$. That is because under that hypothesis there is $m_0\le m$ such that the hypothesis of the lemma is completely satisfied with $m_0$, hence there is a polynomial $p_0$ of degree $\le m_0-1$ such that the conclusion of the lemma holds for $m_0$. Then we can set $p(z)=(z-w)^{m-m_0}p_0(z)$ and use the identity $D^{m_0}_w(p_0f)=D_w^{m}(pf)$ to deduce the claim.
\begin{proof} First we show that $\la (z-w)^{m-1},f_0\ra_{\HH(B)}\ne 0$. The polynomials are dense in $\HH(B)$ (see \cite{AlemanMalman}), hence by the hypothesis there is a polynomial $q$ such that $\la q, (T^*-\overline{w})^{m-1}f_0\ra \ne 0$. The function $r(z)= \frac{q(z)-q(w)}{z-w}$ is a polynomial and satisfies $q(z)=q(w) +(z-w)r(z)$, hence
\begin{align*} 0&\ne \la q, (T^*-\overline{w})^{m-1}f_0\ra \\
&= q(w) \la (z-w)^{m-1},f_0\ra_{\HH(B)} +\la (z-w)r, (T^*-\overline{w})^{m-1}f_0\ra\\
&= q(w) \la (z-w)^{m-1},f_0\ra_{\HH(B)} + \la r, (T^*-\overline{w})^{m}f_0\ra\\
&= q(w) \la (z-w)^{m-1},f_0\ra_{\HH(B)}\end{align*}
Hence $\la (z-w)^{m-1},f_0\ra_{\HH(B)}\ne 0$.

Now set $p(\lambda)=\sum_{j=0}^{m-1} \la (z-w)^{m-1-j},f_0\ra_{\HH(B)} (\lambda-w)^j$. Then $p$ is a polynomial of degree $\le m-1$ and $p(w)\ne 0$.

Let $g$ be a polynomial, then $g(z)= \sum_{j\ge 0} \frac{g^{(j)}(w)}{j!}(z-w)^j$ and since $\la (z-w)^k,f_0\ra_{\HH(B)}=0$ for all $k \ge m$ we have
\begin{align*} \la g,f_0\ra_{\HH(B)}&=  \sum_{j= 0}^{m-1} \frac{g^{(j)}(w)}{j!}\la (z-w)^j, f_0\ra_{\HH(B)}\\
&=  \sum_{j= 0}^{m-1} \frac{g^{(j)}(w)}{j!}\frac{p^{(m-1-j)}(w)}{(m-1-j)!}\\
&= \frac{(pg)^{(m-1)}(w)}{(m-1)!}.\end{align*}
Now Lemma \ref{DeltaLocalDiri} implies that $D_w^m(zpg)-D_w^m(pg)=|\la g, f_0\ra_{\HH(B)}|^2$.

Finally, for $|\lambda|=1$ we apply this formula with $g(z)= \frac{f(z)-f(\lambda)}{z-\lambda}$, and integrate  over the unit circle with respect to $\frac{|d\lambda|}{2\pi}$. The lemma then follows from Lemma \ref{AlemanMalmanFormulaLocalDiri}.
\end{proof}

\begin{theorem}\label{IntersectionOfBasic} Let $m\in \N$, $B\in \HS(\C^n,\C)$ be such that $B(0)=0$ and  $T=(M_z,\HH(B))$ is a bounded operator. Write $\Delta=T^*T-I$, $\HN=[\ran \Delta]_{T^*}$,  $A=P_{\HN}T|\HN$, and assume that
\begin{itemize}
\item $\dim \HN<\infty$,
\item $\sigma(A)=\{w_1, \dots w_k\}$, and
\item $p(z)= \prod_{j=1}^k (z-w_j)^{m_j}$ is the characteristic polynomial of $A$.\end{itemize}

 If $T$ is a strict $2m$-isometry, then $$\sigma(A) \subseteq \T,\ \ m=\max \{m_j: 1\le j\le k\}, $$ and there are $n_1, \dots, n_k\in \N$ with $n_j\le m_j$ for all $j$ and $\sum_{j=1}^k n_j=\rank \Delta$, and there are polynomials $\{p_{ij}\}_{1\le j\le k, 1\le i\le n_j}$ with $p_{1j}(w_j)\ne 0$ for all $j$, such that the degree of $p_{ij}$ is $\le m_j-1$ for all $j$ and $i$, and such that
$$\|f\|^2_{\HH(B)}=\|f\|^2_{H^2}+ \sum_{j=1}^k \sum_{i=1}^{n_j} D^{m_j}_{w_j}(p_{ij}f)$$ for all $f\in \HH(B)$.
\end{theorem}

\begin{proof} The hypothesis that $\dim \HN <\infty$ implies that $B$ is rational (see Theorem \ref{characterizationmate}). Since $B\in \HS(\C^n,\C)$ we have $\dim \ker T^*=1$. Thus, if $T$ is a strict $2m$-isometry, then $\sigma(A)\subseteq \T$ and $m=\max \{m_j: 1\le j\le k\}$ by Theorem \ref{characterization1} and  Remark 2 following it.  Also, since $T$ is a $2m$-isometry, Theorem \ref{characterizationR} implies that $\Delta=\sum_{j=1}^k \Delta_j$, where each $\Delta_j\ge 0$ and $\ran \Delta_j \subseteq \HN_{\overline{w}_j}(A^*)$. Let $n_j=\rank \Delta_j$, then $n_j \le \dim \HN_{\overline{w}_j}(A^*)=m_j$. This is because  the characteristic polynomial of $A$ equals its minimal polynomial (see Lemma \ref{eigenspaces of M^*}).

Let $1\le j\le k$. Since $\ran \Delta_j \subseteq \HN_{\overline{w}_j}(A^*)$ there are  $f_{1j}, \dots, f_{n_jj}\in  \HN_{\overline{w}_j}(A^*)$ such that $\Delta_j=\sum_{i=1}^{n_j} f_{ij}\otimes f_{ij}$. For fixed $j$ we have $(T^*-\overline{w}_j)^{m_j}f_{ij}=0$ for all $i$, and we may assume that $(T^*-\overline{w}_j)^{m_j-1}f_{1j}\ne 0$. Indeed, if  $(T^*-\overline{w}_j)^{m_j-1}f_{ij}= 0$ for all $i$, then one shows that $\frac{p(z)}{z-w_j}$ would annihilate $A$.   This contradicts that  the minimal and characteristic polynomials must agree. Now we use Lemma \ref{rank1andLocalDirichlet} and the remark immediately following it to deduce that there are polynomials
$\{p_{ij}\}_{1\le j\le k, 1\le i\le n_j}$ with $p_{1j}(w_j)\ne 0$ for all $j$, such that the degree of $p_{ij}$ is $\le m_j-1$ for all $j$ and $i$ and such that for each $j$ and $i$ we have
$$D^{m_j}_{w_j}(p_{ij}f)= \int_{|\lambda|=1}\left|\la \frac{f-f(\lambda)}{z-\lambda}, f_{ij}\ra_{\HH(B)}\right|^2 \frac{|d\lambda|}{2\pi}$$
for every polynomial $f$.

Write $D= \Delta^{1/2}$. Then since $\Delta= \sum_{j=1}^k\sum_{i=1}^{n_j} f_{ij}\otimes f_{ij}$ we have for every polynomial $g$ that
$$\|Dg\|^2_{\HH(B)}= \sum_{j=1}^k\sum_{i=1}^{n_j} |\la g,f_{ij}\ra_{\HH(B)}|^2.$$ Hence Lemma \ref{AlemanMalmanFormula} implies
\begin{align*} \|f\|^2_{\HH(B)} &= \|f\|^2_{H^2} + \sum_{j=1}^k\sum_{i=1}^{n_j}\int_{|\lambda|=1}  |\la \frac{f-f(\lambda)}{z-\lambda},f_{ij}\ra_{\HH(B)}|^2 \frac{|d\lambda|}{2\pi}\\
&=\|f\|^2_{H^2} + \sum_{j=1}^k\sum_{i=1}^{n_j} D^{m_j}_{w_j}(p_{ij}f)
\end{align*} for every polynomial $f$. The theorem follows since  the polynomials are dense in $\HH(B)$.
\end{proof}

In order to establish the missing parts of the proof of Theorem \ref{MainTheorem} we prove the following converse of Theorem \ref{IntersectionOfBasic}.

\begin{theorem}\label{MainThmConverse} Let $w_1, \dots , w_k$ be mutually distinct points in $\T$ , let $m_1, \dots, m_k, n_1, \dots n_k \in \N$, set $\tilde{n}_j=\min (m_j,n_j)$,  and let $\{p_{ij}\}_{1\le j \le k, 1\le i \le n_j}$ be polynomials such that the degree of $p_{ij}$ is $\le m_j-1$   and such that $p_{1,j}(w_j) \ne 0$ for each $j$.

Then there is a $n \in \N$, $n \le \sum_{j=1}^k\tilde{n}_j$, and a rational $B\in \HS(\C^n,\C)$ such that $B(0)=0$ and  $$\|f\|^2_{\HH(B)}= \|f\|^2_{H^2}+ \sum_{j=1}^k \sum_{i=1}^{n_j} D_{w_j}^{m_j}(p_{ij}f)$$ for all $f\in \HH(B)$.

Furthermore, if $m= \max\{m_j: 1\le j\le k\}$, then $T=(M_z,\HH(B))$ is an expansive $2m$-isometry with $\rank \Delta \le n$. If $\HN=[\ran \Delta]_{T^*}$ and $A=P_{\HN}T|\HN$, then $\sigma(A)=\{w_1, \dots w_k\}$ and the characteristic polynomial of $A$ is $p_A(z)=\prod_{j=1}^k(z-w_j)^{m_j}$. Furthermore, the mate $a$ of $B$ is of the form $a(z)= \frac{p_A(z)}{q(z)}$, where $q$ is a polynomial of degree $\le \sum_{j=1}^k m_j$ and it has no zeros in the closed unit disc. $q$ is determined by \begin{align}\label{q}|q(z)|^2 = |p_A(z)|^2 + \sum_{j=1}^k \left|\frac{p_A(z)}{(z-w_j)^{m_j}}\right|^2 \sum_{i=1}^{n_j}|p_{ij}(z)|^2 \ \text{ for all }|z|=1.\end{align}
\end{theorem}
\begin{proof} For $1\le j\le k$ and $1\le i \le n_j$ define norms by $$\|f\|^2_{ij}= \|f\|^2_{H^2}+D^{m_j}_{w_j}(p_{ij}f),$$
$$\|f\|^2_j= \|f\|^2_{H^2}+ \sum_{i=1}^{n_j} D_{w_j}^{m_j}(p_{ij}f),$$
and define $$\|f\|^2=\|f\|^2_{H^2}+ \sum_{j=1}^k \sum_{i=1}^{n_j} D_{w_j}^{m_j}(p_{ij}f).$$ Furthermore, let $\HH_{ij}, \HH_j,$ and $\HH$ be the Hilbert function spaces  that consist of all analytic functions such that the corresponding norms are finite. Then by Theorem \ref{Thm1.1again} each operator $T_{ij}=(M_z,\HH_{ij})$ is a bounded, analytic, and expansive $2m_j$-isometry with $\dim \ker T_{ij}^*=1$. It is then immediately clear that each of the operators $T_j=(M_z,\HH_j)$ and $T=(M_z,\HH)$ are bounded, expansive, and analytic. We also note that $\HH_j =\bigcap_{i} \HH_{ij}$ and $\HH=\bigcap_{j}\HH_j$.
By Lemma 2.1 of \cite{Ri87} the condition $\dim \ker T_{ij}^*=1$ is equivalent to $f\in \HH_{ij}, f(0)=0 \Rightarrow f(z)/z\in \HH_{ij}$. Thus, we may  apply that lemma to $T_j$ and $T$ and conclude that $\dim \ker T^*_{j}=\dim \ker T^*=1$ for all $j$.
Note that an operator $(M_z,\HK)$ is a $2M$-isometry, if and only if
$$\sum_{r=0}^{2M}\binom{2M}{r} (-1)^r \|z^r f\|^2_{\HK}=0$$ for all $f\in \HK$. Thus, it is clear that for each $j$  $T_j$ is a $2m_j$-isometry and that $T$ is a $2m$-isometry, where $m=\max\{m_j: 1\le j\le k\}$. Furthermore, if $\Delta=T^*T-I$, $\Delta_j=T^*_jT_j-I$, then by Lemma \ref{DeltaLocalDiri} we have
$$\la \Delta f,f\ra_{\HH}= \sum_{j=1}^k\sum_{i=1}^{n_j} |\frac{(p_{ij}f)^{(m_j-1)}(w_j)}{(m_j-1)!}|^2$$ for all $f\in \HH$ and
$$\la \Delta_jf,f\ra_{\HH_j}= \sum_{i=1}^{n_j} |\frac{(p_{ij}f)^{(m_j-1)}(w_j)}{(m_j-1)!}|^2$$ for all $f \in \HH_j$. This implies that $\rank \Delta_j\le n_j$ and $ \rank \Delta \le \sum_{j=1}^kn_j<\infty$. Thus, there are rational and scalar-valued Schur functions $B_j, B$ such that $T_j$ is unitarily equivalent to $(M_z,\HH(B_j))$ and $T$ is unitarily equivalent to $(M_z,\HH(B))$.
Furthermore, one similarly observes  that $\Delta_j (T_j-w_j)^{m_j}=0$, hence by taking the adjoint and using Lemma \ref{eigenspaces of M^*} we see that $\dim \ran \Delta_j \le m_j$. Thus, we have $\rank \Delta_j \le \tilde{n}_j$ for each $j$.

Now for each $j$ let $J_j$ the inclusion map of $\HH\subseteq \HH_j$, then $\Delta=\sum_{j=1}^k J_j^*\Delta_jJ_j$. Thus, we must have $n=\rank \Delta \le \sum_{j=1}^k\tilde{n}_j$. That means that by Theorem \ref{expandingOp} we may choose $B\in \HS(\C^n,\C)$ such that $B(0)=0$ and $(M_z,\HH(B))$ is unitarily equivalent to $T$. As usual, since $\dim \ker T^*=1$, the unitary operator must be given by multiplication by a  function $G$, and the condition $B(0)=0$ is equivalent to $\ker M_z^*=\ker T^*$ being equal to the constant functions. Thus, the multiplier $G$ must be constant, and hence $\HH=\HH(B)$ and
$\|f\|_{\HH(B)}=\|f\|$ for all $f\in \HH$.

Next we use the definitions  $\HN=[\ran \Delta]_{T^*}$ and $A=P_{\HN}T|\HN$, and $p_A(z)=\prod_{j=1}^k(z-w_j)^{m_j}$. We will show that $p_A$ is the minimal polynomial of $A$. That will show that $\sigma(A)=\{w_1,\dots, w_k\}$ and that $p_A$ is the characteristic polynomial of $A$ (by Lemma \ref{eigenspaces of M^*}).

Note that if $p$ is any polynomial, then $p(A)=0$, if and only if $\ran p(T) \perp T^{*n} \Delta y$ for all $n\ge 0$ and all $y\in \HH$. That condition is equivalent to $\Delta p(T)=0$. Of course, a similar statement holds for each $T_j$.

Fix $1\le j\le k$. The condition $p_{1j}(w_j)\ne 0$ implies by Theorem \ref{Thm1.1again} that $T_{1j}$ is a strict $2m_j$-isometry, i.e. $\beta_{2m_j-1}(T_{1j})\ne 0$. For each $i$ the fact that $T_{ij}$ is a $2m_j$-isometry, implies that $\beta_{2m_j-1}(T_{ij}) \ge 0 $ (see \cite{AS956}), hence one easily checks that  $\beta_{2m_j-1}(T_j)\ne 0$. This means that $T_j$ must be a strict $2m_j$-isometry with $\Delta_j(T_j-w_j)^{m_j}=0$, but $\Delta_j(T_j-w_j)^{m_j-1}\ne 0$ (see Corollary \ref{defectrankone}). Then for any polynomial $r$ we have $$\Delta_j(T_j-w_j)^{m_j-1}(r(T_j)-r(w_j))=0$$ and hence if $r(w_j)\ne 0$, then $$\Delta_j(T_j-w_j)^{m_j-1}r(T_j)\ne 0.$$

Then we have
$\Delta p_A(T)=\sum_{j=1}^k J_j^* \Delta_j p_A(T_j)J_j=0$ since for each $j$ we have $\Delta_j(T_j-w_j)^{m_j}=0$. Thus, $p_A(A)=0$.

For each $j$ we can factor $p_A(z)=(z-w_j)^{m_j} p_j(z)$ for some polynomial $p_j$ with $p_j(w_j)\ne 0$, but $\Delta_i p_j(T_i)=0$ for all $i\ne j$.
Thus,
\begin{align*}\Delta (T-w_j)^{m_j-1}p_j(T)&=\sum_{i=1}^k J_i^*\Delta_i (T_i-w_j)^{m_j-1}p_j(T_i)J_i\\
&=J_j^* \Delta_j(T_j-w_j)^{m_j-1}p_j(w_j)J_j\ne 0,
\end{align*}
since $J_j$ is 1-1. This proves that $p_A$ is the minimal polynomial of $A$.

As mentioned above we conclude that $p_A$ is the characteristic polynomial of $A$, and hence $N=\dim \HN= \deg p_A=\sum_{j=1}^km_j$. Hence by Theorem \ref{characterizationmate} we conclude that the degree of $B$ is $N$. Then Theorem \ref{characterizationmate} also implies that $a(z)=p_A(z)/q(z)$ for some polynomial $q$ of degree $\leq N$ and which has no zeros in $\overline{\D}$.

Finally, let $\HM=\HN^\perp$ and let $\varphi\in \HM\ominus z\HM$, $\|\varphi\|=1$. Then by Remark \ref{equalityDegrees} we may assume that $\varphi=a$. Note that since for each $j$ the function $a$ has a zero of multiplicity $m_j$ at $w_j$ we observe that for each $j$ and  $i$ the Taylor polynomial at $w_j$  $T_{m_j-1}(p_{ij}a,w_j)(z)=0$ for all $z$.
This implies that for all $n \in \N$
\begin{align*}0=\la z^n \varphi, \varphi\ra_{\HH(B)}&= \int_{|z|=1} z^n |a|^2 \left(1+ \sum_{j=1}^k\sum_{i=1}^{n_j} \frac{|p_{ij}|^2}{|z-w_j|^{2m_j}} \right) \frac{|dz|}{2\pi}.
\end{align*}
Hence $$|a(z)|^{-2}= 1+ \sum_{j=1}^k\sum_{i=1}^{n_j} \frac{|p_{ij}(z)|^2}{|z-w_j|^{2m_j}}$$ for a.e. $z\in \T$. Considering that $a=p_A/q$ this shows that $q$ satisfies equation (\ref{q}).

This last calculation reverses. If a polynomial $q_0$ of degree $\le N$ satisfies (\ref{q}), then the function $\psi = p_A/q_0$ satisfies $\la z^n \psi, \psi\la_{\HH(B)}=0$ for all $n\in \N$, and it has norm 1. Since $p_A(A)=0$ we must have that $\psi\in \HM$. Then since $\HM=\varphi H^2$ we obtain $\psi= \varphi f$ for some $f\in H^2$, and by the orthogonality condition $f$ would have to be an inner function. Using the definitions of $\psi$ and $\varphi=a$ we see that $q_0=q/f$, and that implies that $f$ must be a constant of modulus 1 ($q$ has no zeros in $\D$).
\end{proof}

\section{A Construction}

Our Theorem \ref{MainTheorem} about $2m$-isometries leaves open the question of what  the precise connection between the polynomials $p_{ij}$ and the Schur function $B$ is. The general construction that we described in Section 4 is not very explicit in this case. In general there may be many different Schur functions $B$  that have the same mate $a$. If $a$ is of the form $a(z)=(z-w)^m/q(z)$, then for any $B$ with mate $a$, with $B(0)=0$, and of degree $m$, it turns out that $T=(M_z,\HH(B))$ is a $2m$-isometry. In fact, in this case Theorem \ref{characterizationmate} implies that $(z-w)^m$ is the characteristic polynomial of $A$, hence $(T^*-\overline{w})^m\Delta =0$, and the conclusion follows from part (b) of Theorem \ref{twonisometry}.  However, if $\sigma(A)$ contains more than one point, then it may happen that two Schur functions $B_1$ and $B_2$ have the same mate $a$, and $B_1$ is a $2m$-isometry, but $B_2$ is not.

Suppose all the conditions are met in Theorem \ref{MainThmConverse}. Then
\begin{align}\label{normcal}
\|f\|^2_{\HH(B)}= \|f\|^2_{H^2}+ \sum_{j=1}^k \sum_{i=1}^{n_j} D_{w_j}^{m_j}(p_{ij}f), f\in \HH(B).
\end{align}
$p_A(z)=\prod_{j=1}^k(z-w_j)^{m_j}$, $a(z)= \frac{p_A(z)}{q(z)}$, where $q$ is a polynomial of degree $\le \sum_{j=1}^k m_j$ and it has no zeros in the closed unit disc. $q$ is determined by (\ref{q}).

We will now indicate how to calculate the reproducing kernel $K_w^B(z)$ by  use of the Grammian of a dual basis to $\{\overline{\partial}^iK_{w_j}^B\}_{ij}$. In the case of 2-isometries this idea was developed in \cite{AgRi90}.

We know that $K^B_w(z)= P_{\HN}K_w^B(z)+\frac{a(z)\overline{a(w)}}{1-z\overline{w}}$, where $\HN$ is the $\sum_{j=1}^k m_j$-dimensional space spanned by the functions $\{\overline{\partial}^iK_{w_j}^B = \frac{\partial^i K_{w_j}^B}{\overline{\partial}w^i}: 0 \leq i \leq m_j-1, 1 \leq j \leq k\}$. We also know that a dual basis to $\{\overline{\partial}^iK_{w_j}^B: 0 \leq i \leq m_j-1, 1 \leq j \leq k\}$ must be of the form $\{f_{ij}: 0 \leq i \leq m_j-1, 1 \leq j \leq k\}$, where $f_{ij}(z) = \frac{a(z)}{(z-w_j)^{m_j-i}}g_{ij}$, and $g_{ij}$ are polynomials with $\deg g_{ij} \leq m_j -1 -i$, and are determined by the conditions that $f_{ij}^{(i)}(w_j)=1, f_{ij}^{(l)}(w_j) = 0, i+1 \leq l \leq m_j -1$. Indeed, these functions are in $\HN$ by (\ref{N as rational functions}), and by the choice of constants they are dual to the functions $\{\overline{\partial}^iK_{w_j}^B\}$. We note that the coefficients of the functions $g_{ij}$ can be calculated as solutions to linear equations (that require knowledge of the polynomial $q$).
Next we form the Grammian matrix
$$F=\left[\begin{matrix} \la f_{ij},f_{st}\ra_{\HH(B)} \end{matrix}\right],$$
where $\la f_{ij},f_{st}\ra_{\HH(B)}$ can be calculated by (\ref{normcal}).

Now note that
$$P_{\HN} = \sum_{j=1}^k \sum_{i=0}^{m_j-1} \overline{\partial}^i K_{w_j}^B \otimes f_{ij} = \sum_{j=1}^k \sum_{i=0}^{m_j-1} f_{ij}\otimes \overline{\partial}^i K_{w_j}^B.$$
Thus,
$$P_{\HN}K_w^B(z)= \sum_{j=1}^k \sum_{i=0}^{m_j-1} \overline{\overline{\partial}^i K_{w_j}^B(w)}f_{ij}(z).$$
Let $[f_{ij}]^t$ be the column vector $[f_{ij}]_{0 \leq i \leq m_j -1, 1 \leq j \leq k}^t$. Then from
$$f_{st} = \sum_{j=1}^k \sum_{i=0}^{m_j-1}  \langle f_{st}, f_{ij}\rangle_{\HH(B)} \overline{\partial}^i K_{w_j}^B, 0 \leq s \leq m_t -1, 1 \leq t \leq k,$$
we obtain that
$\left[\begin{matrix} f_{ij} \end{matrix}\right]^t= F \left[\begin{matrix} \overline{\partial}^i K_{w_j}^B\end{matrix}\right]^t$ and hence $\left[\begin{matrix} \overline{\partial}^i K_{w_j}^B\end{matrix}\right]^t= F^{-1} \left[\begin{matrix} f_{ij} \end{matrix}\right]^t$.

Now we have determined $f_{ij}$ and $\overline{\partial}^i K_{w_j}^B$. So we obtain the formula for the reproducing kernel $K^B_w(z)= P_{\HN}K_w^B(z)+\frac{a(z)\overline{a(w)}}{1-z\overline{w}}$.

We use the following example to illustrate the above construction.
\begin{example} Define a norm $2$-isometric norm by
$$\|f\|^2=\|f\|^2_{H^2}+\frac{9}{16}(D_1(f)+D_{-1}(f)).$$
Then $p_A(z)=z^2-1$, $a=p_A/q$, where for $|z|=1$ $$|q(z)|^2= |z^2-1|^2+\frac{9}{16} (|z+1|^2+|z-1|^2)= \frac{17}{4} -2 \text{ Re } z^2.$$ Thus $q(z)= \frac{z^2-4}{2}$, and $a(z)= \frac{2(z^2-1)}{z^2-4}$.

Note that no other $2$-isometry will have the same $a$. Indeed, the only possible candidates would have to have norm $\|f\|^2_*=\|f\|^2_{H^2}+c_1D_1(f)+c_2D_{-1}(f),$ and unless $c_1=c_2=9/16$ that would lead to a different $q$. Thus, there is a unique space $\HH(B)$ such that $(M_z, \HH(B))$ is a $2$-isometry and $B$ has mate $a$.

The function $b(z)=\frac{3z}{z^2-4}$ satisfies $b(0)=0$, has degree 2, and has mate equal to $a$, but by Theorem \ref{2n-iso} $(M_z,\HH(b))$ cannot be $2$-isometric. That is because in our example the rank of $\Delta$ equals 2. So the reproducing kernel for the given norm has to be of the form $K^B_w(z)$ for $B=(b_1,b_2)$ for linearly independent functions $b_1$ and $b_2$ with $B(0)=0$ and such that the degree of $B$ is 2.

\

Now we use the above construction to find $K_w^B(z)$. We have $\HN$ is the $2$-dimensional space spanned by the kernel functions $K_1^B$ and $K_{-1}^B$. Then the dual basis to $\{K_1^B, K_{-1}^B\}$ is $\{f_1, f_2\}$, where $f_1(z) = \frac{-3}{2}\frac{z+1}{z^2-4}, f_2(z) = \frac{3}{2}\frac{z-1}{z^2-4}$. By calculation, we obtain
$$F=\left[\begin{matrix} \la f_1,f_1\ra & \la f_1,f_2\ra\\\la f_2,f_1\ra & \la f_2,f_2\ra \end{matrix}\right]= \left[\begin{matrix}  \frac{21}{32} & \frac{-9}{32}\\ \frac{-9}{32} & \frac{21}{32} \end{matrix}\right], \quad F^{-1}= \left[\begin{matrix}  \frac{28}{15} & \frac{12}{15}\\ \frac{12}{15} & \frac{28}{15} \end{matrix}\right].$$
Then $$K_1^B(z) = \frac{-8/5z-4}{z^2-4}, \quad K_{-1}^B(z) = \frac{8/5z-4}{z^2-4}.$$
So
\begin{align*}
K^B_w(z) & = \frac{1-B(z)B(w)^*}{1-z\overline{w}} = P_{\HN}K_w^B(z)+\frac{a(z)\overline{a(w)}}{1-z\overline{w}}\\
& = \overline{K_1^B(w)}f_1(z)+\overline{K_{-1}^B(w)}f_2(z) + \frac{a(z)\overline{a(w)}}{1-z\overline{w}}\\
& = \frac{1-(z\overline{w})(9/5z\overline{w} + 36/5)/((z^2-4)(\overline{w}^2-4))}{1-z\overline{w}}.
\end{align*}
There are many choices for $B$ which give the same expression $K_w^B(z)$, for example, we can take $B = (b_1, b_2), b_1(z) = \frac{6/\sqrt{5}z}{z^2-4}, b_2(z) = \frac{3/\sqrt{5}z^2}{z^2-4}$.

\

Since this example only involves two points $1$ and $-1$, we have another way to derive the reproducing kernel $K_w^B(z)$. Suppose $B(z) = \frac{z(v_1 z + v_2)}{q(z)}, v_i = (\alpha_i, \beta_i), i = 1, 2$. Note that $a(z)= \frac{2(z^2-1)}{z^2-4}$, and $|B(z)|^2 + |a(z)|^2 = 1, z \in \T$. We obtain that
$\langle v_1, v_2\rangle = 0$, $\|v_1\|^2 + \|v_2\|^2 = 9$. So
\begin{align*}
K_w^B(z)& = \frac{1-B(z)B(w)^*}{1-z\overline{w}}\\
&= \frac{1-(z\overline{w})((9-\|v_2\|^2)z\overline{w} + \|v_2\|^2)/((z^2-4)(\overline{w}^2-4))}{1-z\overline{w}}.
\end{align*}
Then using $1 = \langle z, \overline{\partial}K_0^B(z)\rangle$, we find $\|v_2\|^2 = \frac{36}{5}$, and
$$K_w^B(z) = \frac{1-(z\overline{w})(9/5z\overline{w} + 36/5)/((z^2-4)(\overline{w}^2-4))}{1-z\overline{w}}.$$
\end{example}

\end{document}